\documentclass[11pt]{amsart}

\usepackage{amsfonts, amstext, amsmath, amsthm, amscd, amssymb, graphicx, epstopdf}
\usepackage{epsfig, graphics, psfrag, enumerate}
\usepackage{color}
\usepackage{verbatim}
\let\oldmarginpar\marginpar
\renewcommand\marginpar[1]{\oldmarginpar[\raggedleft\footnotesize #1]%
{\raggedright\footnotesize #1}}

\newtheorem{theorem}{Theorem}[section]

\newtheorem{lemma}[theorem]{Lemma}
\newtheorem{corollary}[theorem]{Corollary}
\newtheorem{proposition}[theorem]{Proposition}
\newtheorem{conjecture}[theorem]{Conjecture}

\newtheorem{define}[theorem]{Definition}

\theoremstyle{definition}
\newtheorem{remark}[theorem]{Remark}
\newtheorem{example}[theorem]{Example}

\newcommand{\ZZ}{{\mathbb{Z}}}

\newcommand{\QQ}{{\mathbb{Q}}}

\newcommand{\bdy}{{\partial}}
\newcommand{\M}{M_{K_{p,q}}}

\newcommand{\cut}{{\backslash \backslash}}

\newcommand{\C}{{C_{p, q}}}

\newcommand{\K}{{K_{p, q}}}

\newcommand{\abs}[1]{{\left\vert #1 \right\vert}}

\newcommand{\GA}{{\mathbb{G}_A}}
\newcommand{\GB}{{\mathbb{G}_B}}
\newcommand{\G}{{\mathbb{G}}}

\newcommand\no[1]{}

\newtheorem*{namedtheorem}{\theoremname}
\newcommand{\theoremname}{testing}

\def\BC{\mathbb C}
\def\BN{\mathbb N}
\def\BZ{\mathbb Z}

\def\BQ{\mathbb Q}

\def\CS{\mathcal S}

\def\la{\langle}
\def\ra{\rangle}

\def\ve{\varepsilon}
\def\be { \begin{equation} }
\def\ee { \end{equation} }

\begin{document}

\title[]{Knot Cabling and the Degree of the Colored Jones Polynomial}

\author[]{Efstratia Kalfagianni}
\address{Department of Mathematics, Michigan State University, East Lansing, MI, 48824}
\email{kalfagia@math.msu.edu}

\author[]{Anh T. Tran}
\address{Department of Mathematical Sciences, University of Texas at Dallas,  Richardson, TX, 75080}
\email{att140830@utdallas.edu}
\bigskip

\bigskip

\begin{abstract} We study the behavior of   the degree of the colored Jones polynomial and the boundary slopes of knots under  the operation of cabling. We show that, under certain hypothesis on this degree, if a knot $K$ satisfies the Slope Conjecture then a $(p, q)$-cable of $K$ satisfies the conjecture,
provided that $p/q$ is not a Jones slope of $K$.
As an application we  prove the Slope Conjecture for iterated cables of adequate knots and for iterated torus knots.
Furthermore we show that, for these knots, the degree of the  colored Jones polynomial also
determines the topology of a surface that satisfies the Slope Conjecture.
We also state a conjecture   suggesting a topological interpretation of 
 the linear terms of the degree of the colored Jones polynomial (Conjecture \ref{conj}), and we prove it for the following classes of knots:
 iterated  torus knots and iterated cables of adequate knots, iterated cables
of several non-alternating knots with up to nine crossings, pretzel knots of type $(-2, 3, p)$ and their cables, and two-fusion knots.


\bigskip

\bigskip

\noindent {2010 {\em Mathematics Classification:} {\rm Primary 57N10. Secondary 57M25.}\\

}
\end{abstract}
\thanks {\today}
\thanks{E. K. was partially supported in part by NSF grants DMS--1105843 and DMS--1404754}

\maketitle

\section{Introduction}

\subsection{The Slope Conjecture} For a knot $K \subset S^3$, let $n(K)$ denote a tubular neighborhood of
$K$ and let $M_K:=\overline{ S^3\setminus n(K)}$ denote the exterior of
$K$. Let $\langle \mu, \lambda \rangle$ be the canonical
meridian--longitude basis of $H_1 (\bdy n(K))$.  An element $a/b \in
{\QQ}\cup \{ 1/0\}$ is called a \emph{boundary slope} of $K$ if there
is a properly embedded essential surface $(S, \bdy S) \subset (M_K,
\bdy n(K))$, such that  $\bdy S$ represents $a \mu + b \lambda \in
H_1 (\bdy n(K))$.  Hatcher showed that every knot $K \subset S^3$
has finitely many boundary slopes \cite{hatcher}.
We will use $bs_K$ to denote the set of boundary slopes of $K$.

For a positive integer $n$, let $J_K(n) \in \BZ[v^{\pm 1/2}]$ be the $n$-th colored Jones polynomial of $K$ with framing 0 \cite{Jo, RT}, normalized so that $$J_{\text{unknot}}(n) = \frac{v^{n/2}- v^{-n/2}}{v^{1/2} -v^{-1/2}}.$$ 
Here $v=A^{-4}$, where $A$ is the variable in the Kauffman bracket \cite{Ka}.

For a sequence $\{x_n\}$, let $\{x_n\}'$ denote the set of its cluster points. Let $d_+[J_K(n)]$ denote the highest degree of $J_K(n)$ in $v$,
and let $d_-[J_K(n)]$ denote the lowest degree. Elements of the sets 
$$js_K:= \left\{ 4n^{-2}d_+[J_K(n)]  \right\}' \quad
 \mbox{and} \quad js^*_K:= \left\{ 4n^{-2}d_-[J_K(n)] \right\}' $$
 are called {\em Jones slopes} of $K$. 
 Garoufalidis   \cite{ga-quasi} showed that every knot has  finitely many Jones slopes.
Furthermore,   he formulated the following conjecture and he  verified it for alternating knots, non-alternating knots with up to nine crossings, torus knots, 
and for the family ($-2, 3, p)$ of 3-string pretzel  knots \cite{Ga-slope}.

\begin{conjecture} [Slope Conjecture] For every knot $K \subset S^3$ we have 
$$(js_K \cup js^*_K) \subset bs_K.$$
\end{conjecture}

 Futer, Kalfagianni and Purcell \cite{FKP}  verified the conjecture
for \emph{adequate knots} (see Definition \ref{defi:adequate} for terminology). The works of Garoufalidis and Dunfield \cite{DG} and
Garoufalidis and van der Veen \cite{GV} verified the conjecture for a certain 2-parameter family of closed 3-braids, called 2-fusion knots. More recently, Lee and van der Veen \cite{leV} have proved the conjecture for several more 3-string pretzel knots.

In this paper we study the behavior of the  boundary slopes and the Jones slopes of knots under the operation of cabling and prove the Slope Conjecture for
  cables of
several classes of knots. We also formulate, and verify for several classes of knots,  conjectures providing  topological interpretations of the linear terms
of the degree of the colored Jones polynomial (Conjectures \ref{slopeaug} and \ref{conj}). To state our results we need some preparation.

\subsection{Cable knots}

Suppose $K$ is a knot with framing 0, and $p,q$ are coprime integers.
The $(p,q)$-cable $K_{p,q}$ of $K$ is the $0$-framed satellite of $K$ with
pattern $(p,q)$-torus knot (see Section \ref{boundary-slopes} for more details).
In the statements of results below, and throughout the paper,  we will assume that our cables are non-trivial in the  sense  that $\abs{q}>1$.

\begin{theorem} \label{bdry slope}
 For every knot $K \subset S^3$ and  $(p, q)$ coprime  integers
we have 
 $$\left( q^2 bs_K \cup \{pq\} \right) \subset bs_{K_{p,q}}.$$
\end{theorem}

To continue, we recall
that  for any knot $K\subset S^3$
the degrees $ d_+[J_{K}(n)] $ and $ d_-[J_{K}(n)] $ are quadratic quasi-polynomials in $n$  \cite{ga-quasi}. This implies that their coefficients are periodic functions $\BN \to \BQ$.
The common period  of $d_+[J_K(n)]$ and $d_-[J_K(n)]$ is called \emph{the period of $K$}, denoted  by $\pi(K)$. 

In this paper we are concerned with the Jones slopes 
for knots with $\pi(K)\leq 2$ and for knots where the leading coefficient  of $d_+[J_K(n)]$ becomes constant for $n$ large enough.
We show that the Jones slopes of these  knots behave similarly to  boundary slopes under cabling operations (Propositions \ref{quasi} and \ref{quasi-constant}).
In particular, for knots with period at most two, combining  our results about Jones slopes with Theorem \ref{bdry slope} we obtain the following.

\begin{theorem}
 \label{period2}
 Let $K$ be a knot such that, for $n \gg 0$,
$$
d_+[J_K(n)]=a(n)n^2+b(n)n+d(n) \  {\rm and } \  d_-[J_K(n)]=a^*(n)n^2+b^*(n)n+d^*(n)$$
are  quadratic quasi-polynomials of period $\le 2$, with $b(n) \le 0$ and $b^*(n) \ge 0$. Suppose $\frac{p}{q} \notin js_K$. 
Then, we have 
$$ js_{\K} \subset  \big( q^2js_{K} \cup \{pq/4\} \big) \ {\rm and} \  js^*_{\K} \subset  \big( q^2js^*_{K} \cup \{pq/4\} \big).$$
Furthermore, if $(js_K \cup js^*_K) \subset bs_K$ we have $(js_{\K} \cup js^*_{\K}) \subset bs_{\K}$.
 \end{theorem}

The proof of Theorem \ref{period2} reveals that the properties that
 $b(n) \le 0$ and $b^*(n) \ge 0$ are preserved under cabling. We conjecture that these properties hold for all knots,
and that   $b(n)$ and $b^*(n)$  detect the presence of essential annuli in the knot complement.  This is stated in Conjecture
\ref{conj}, which we have verified for all the knots for which the degrees $d_+[J_K(n)]$ and $d_-[J_K(n)]$
are known.

Note that since we take  $\abs{q}>1$, the hypothesis that $p/q$  is not a Jones slope of $K$ will automatically be satisfied  for knots that have all of their Jones slopes
 integers.
A large class of knots with integer Jones slopes is the class of \emph{adequate} knots, which
includes alternating knots, Montesinos knots of length at least four,  pretzel knots with at least four strings
and Conway sums of strongly alternating tangles. The class of \emph{semi-adequate} knots (knots that are
$A$-- or  $B$--adequate) is much broader including all but a handful of prime knots up to 12 crossings, all Montesinos and pretzel knots, positive knots, torus knots, and closed 3-braids. 
The reader is referred to Section 3 below for the precise definition (Definition \ref{defi:adequate}) and to  \cite{fkp:survey, FKP-guts, FKP-semi, adam} and references therein
for more details and examples of semi-adequate knots.

\begin{theorem}
\label{main:js} Let $K$ be a knot and let  $K'$ be an iterated cable knot of $K$. 
\begin{enumerate}
\item If $K$ is a $B$-adequate knot, then $js_{K'} \subset bs_{K'}$.

\item If $K$ is an $A$-adequate knot, then $js^*_{K'} \subset bs_{K'}$.
\end{enumerate}
Hence, if $K$ is an adequate knot then $K'$ satisfies the Slope Conjecture.
\end{theorem}

An iterated torus knot is an iterated cable of the trivial knot.  As a corollary of Theorem \ref{main:js} we have the following.
\begin{corollary} \label{torusi}
Iterated torus knots satisfy the Slope Conjecture. 
\end{corollary}

Theorem \ref{period2} also  applies to several non-alternating  prime knots with up to nine crossings
(see Corollary \ref{low}). We should mention that Motegi-Takata \cite{MoTa} used Theorem  \ref{period2}
to  generalize Corollary \ref{torusi} to all  knots of zero Gromov norm (graph knots).

The proofs of Theorems  \ref{bdry slope} and \ref{period2} reveal that there is a remarkable similarity 
in the behaviors, under cabling,  of the  linear terms of $d_+[J_K(n)]$ and $d_-[J_K(n)]$ and the Euler characteristic of the essential surfaces
``selected" by the Slope Conjecture.  We conjecture that
the cluster points of the sets $\{2b_K(n)\}$ and $\{2b_K^{*}(n)\}$ contain information about the topology of essential surfaces
 that satisfy the Slope Conjecture for $K$. 
To state the conjecture, 
let  $\ell d_+[J_K(n)]$ denote the linear term of $d_+[J_K(n)]$ and let
$$jx_K:= \left\{ 2n^{-1}\ell d_+[J_K(n)]  \right\}'=\{2b_K(n)\}'.$$
\begin{conjecture} [{\rm Strong Slope Conjecture}] \label{slopeaug} Let $K$ be a knot and $a/b\in js_K$, with $b>0$ and $\gcd(a, b)=1$,  a  Jones slope of $K$. Then there is an essential surface $S\subset M_K$, with $\abs{\partial S}$  boundary components, 
and such that each component of $\partial S$ has slope $a/b$ and
$$\frac{\chi(S)}{{\abs{\partial S} b}} \in jx_K.$$
\end{conjecture}

Conjecture \ref{slopeaug}  implies a similar statement for $\{2b^*_K(n)\}'$ since it is know that
$d_-[J_K(n)] =d_+[J_K^*(n)]$, where $K^*$ denotes the mirror image of $K$.
An immediate corollary of Theorem \ref{slopeadequate} is the following:

\begin{corollary} Iterated cables of adequate knots satisfy the Strong Slope Conjecture. In particular, iterated torus knots satisfy the Strong Slope Conjecture.
\end{corollary}

We  also prove Conjecture \ref{slopeaug}  for 
pretzel knots of type $(-2, 3, p)$ and all
Montesinos knots with up to nine crossings (see Section 5).

\subsection{\bf Organization.} The paper is organized as follows. In Section 2 we study boundary slopes of cable knots and we prove Theorem \ref{bdry slope}.  In Sections 3 and 4 we study the behavior  of the degree of the colored Jones polynomial  under knot cabling. In particular,  we discuss cables
of knots of period at most two and we prove Theorem \ref{period2}. In fact, the proof of this theorem allows us to describe explicitly how the Jones slopes of the cable knot $\K$ are related to those of the original knot $K$. We apply our results on knots of period at most two to adequate knots to prove Theorem \ref{main:js}.  
In Section 5 we state, and partially verify, some conjectures about the degree of the colored Jones polynomial. Finally in Section 6 we verify Conjecture \ref{conj} for two-fusion knots.

\section{Boundary Slopes of Cable Knots}\label{boundary-slopes}
 In this section we study how the boundary slopes of knots in $S^3$ affect the boundary slopes 
of their cables. The main result is Theorem \ref{main:bs} that implies in particular  Theorem \ref{bdry slope} stated in the Introduction.
Theorem \ref{main:bs} and Corollary \ref{integers} are key ingredients in the proofs of the results of the paper concerning
relations of the colored Jones polynomial to essential surfaces.  
\smallskip

\subsection{Preliminaries and statement of main result} Let $V$  be a standardly embedded solid torus in $S^3$ and let $V'\subset V$ a second standard solid torus 
that is concentric to $V$. On $\partial V'$  we choose a pair of meridian and canonical longitude  (which also determines such a pair on $\partial V$).  For coprime integers $p,q$, let  $T_{p,q}\subset \partial V'$ be a simple closed curve
of slope $p/q$; that is a $(p, q)$-torus knot.

Recall that for a knot  $K$, $n(K)$  denotes  a neighborhood of $K$.  Embed $V$ in $S^3$ by a  homeomorphism   $f: V\longrightarrow n(K)$
that preserves the canonical longitudes. The $(p,q)$-cable of $K$ is the image
$K_{p,q}:= f(T_{p, q})$.  
The space $C_{p, q}: =f(\overline{V\setminus n(T_{p, q})})$
is called a  $(p,q)$-\emph{cable space}.  The complement of  $K_{p,q}$, denoted by $\M$,
is obtained from the complement of $K$ by attaching $\C$.
The space $\C$ has two boundary components; the inner one $T_{-}=f( \partial n(T_{p, q}))=\partial\M$
and the outer one $T_{+}=f(\partial V)=\partial M_K$.

\begin{define}  \label{bdryslopes} For a knot $K\subset S^3$,  let $\langle \mu, \lambda \rangle$ be the canonical
meridian--longitude basis of $H_1 (\bdy n(K))$.  For a pair of integers $(a, b)$,  the ratio  $a/b \in
{\QQ}\cup \{ 1/0\}$ is called a \emph{boundary slope} of $K$ if there
is a properly embedded essential surface $(S, \bdy S) \subset (M_K,
\bdy n(K))$, such that  $\bdy S$ represents  $a \mu + b \lambda \in
H_1 (\bdy n(K))$. 
\end{define}

In Definition \ref{bdryslopes},  $a, b$  do not need to be coprime.
In fact if $d=\gcd(a, b)$ then we have a surface $S$ as above with $d$ boundary components.
To stress this point sometimes we will say that the \emph{total} slope of $\partial S$ is $a/b$.
Recall that every knot $K$ has finitely many boundary slopes and that
$bs_K$ denotes the set of boundary slopes of $K$.
The rest of the section is devoted to the proof of the following theorem.

\begin{theorem} 
 \label{main:bs} 
 (a)\ Let $K \subset S^3$  be a non-trivial knot and $(p, q)$ coprime  integers.
If $a/b$ is a boundary slope of $K$, then
$q^2 a/b$ is a boundary slope of ${K_{p,q}}$. 
\smallskip

\smallskip

(b)\ For every knot $K \subset S^3$ and $(p, q)$ coprime  integers,
we have 
 $$\left( q^2 bs_K \cup \{pq\} \right) \subset bs_{K_{p,q}}.$$
\end{theorem}

\smallskip


The reader is referred to 
\cite{Ha, hempel} for basic definitions and terminology. 
Let $(p, q)$ be coprime integers.
The cable space $\C$ is a Seifert fibered manifold over an annulus $B$, with one singular fiber of multiplicity $q$. 
In $\C$ there is an essential annulus $A$, that is vertical with respect to the fibration,  with $\bdy A \subset T_{-}$ and with boundary slope equal
to $pq$; this annulus is the \emph{cabling annulus}. 
There are two additional essential annuli in $\C$. One with both boundary components on $T_{-}$, each with slope
$p/q$. The other annulus 
 $A'$ has one component of $\bdy A'$ on $T_{-}$ with slope $pq$, and the second component of $\bdy A'$
on $T_{+}$ with slope $p/q$. See \cite{GL}. In particular, we have the following.
\begin{lemma} \label{cannulus} For any cable knot $\K$, $s=pq$ is a boundary slope in $\M$.

\end{lemma}

Recall that a properly embedded surface $S$ in a 3-manifold  $M$ with boundary, is \emph{essential}
if the map on $\pi_1$ induced by inclusion is injective. If $S$ is orientable
this is equivalent to saying that $S$ is \emph{incompressible} and \emph{ $\partial$-incompressible}
in $M$. If $S$ is non-orientable, then $S$ being essential is  equivalent 
to saying that  the surface $\widetilde{S}: =\partial(S\times I )$ is incompressible and $\partial$-incompressible 
in  $M$.

We need the following lemma, a proof of which is given, for example, in   \cite[Proposition 1.1]{KS}.

\begin{lemma} \label{basic} Let $M$ be a knot complement in $S^3$ and let $\Sigma$ be a properly embedded essential surface 
in $M$. Suppose that a path $\alpha \subset \Sigma$ that has its endpoints on $\partial \Sigma$ is homotopic relative endpoints in $M$
to a path in $\partial M$. Then $\alpha$ is homotopic relative endpoints in $\Sigma$ 
to a path in $\partial \Sigma$.

\end{lemma}

The complement $\M$ of the cable knot  $\K$ is obtained by gluing $\C$ and the complement of $K$
along the torus $T_+$.  If $K$ is a non-trivial knot,  then the  torus $T_{+}$ is essential in $\M$; that is a companion
of $\K$.

We will use $\M \cut T_{+}$ to denote the 3-manifold obtained by splitting
$\M$ along $T_{+}$. Also given a  properly embedded surface $S$ in $\M$ we will use $S\cut T_{+}$
to denote the image of $S$ in $\M \cut T_{+}$.

\begin{lemma}\label{components} Let $K \subset S^3$  be a non-trivial knot and $(p, q)$ coprime  integers.
Let $S$ be a properly embedded surface in $\M$. Suppose that each component of  $S \cut T_{+}$ is essential in the component of $\M \cut T_{+}$ it lies in. Then $S$ is essential in $\M$.

\end{lemma}

\begin{proof} Since $S$ may be non-orientable we will work with the orientable double
$\widetilde {S}= \partial (S \times I)$.  By way of contradiction, suppose that  
$\widetilde {S}$  is not $\pi_1$-injective. This means that $\widetilde {S}$  is either compressible or
$\partial$-compressible. Since $\partial \M$ consists of tori, incompressibility implies $\partial$-incompressibility
\cite{Ha}. Thus we may assume that there is a compression disk $(E, \partial E)\subset (\M, \widetilde {S})$. Since each component of $\widetilde {S}\cut T_{+}$ is incompressible, the intersection $E\cap T_{+}$
must be non-empty. Since $T_{+}$ is essential (and thus incompressible in $\M$) we may eliminate the closed components of $E\cap T_{+}$. Thus we may assume that each component of $E\cap T_{+}$
is an arc properly embedded in $E$. By further isotopy of $E$, during which $\partial E$  moves on
$\widetilde {S}$, we may assume that the intersection $E\cap T_{+}$ is minimal. Now let $\alpha$ be a component
of $E\cap T_{+}$ that is outermost on $E$: It cuts off a disc $E'\subset E$ whose interior contains no further intersections 
with $T_{+}$. Now $\partial E'$ consists of $\alpha$ and an arc $\beta$ that is properly embedded on a component,
say $\Sigma$,
of $\widetilde {S}\cut T_{+}$.  We can use  disk $E'$ to isotope  $\beta$, relatively $\partial \beta$,
on $T_{+}$; this isotopy takes place in  $\M \cut T_{+}$.
Since $\Sigma$ is essential in the component of $\M \cut T_{+}$ 
it lies in, we conclude that the arc $\beta$ may be isotopied on $\Sigma$, relatively $\partial \beta$,
to an arc on $\partial \Sigma$.  This follows from Lemma \ref{basic}. But this isotopy will reduce the components of the intersection  $E\cap T_{+}$, contradicting our assumption of minimality. Thus  $\widetilde {S}$
must be incompressible and therefore, by above discussion,  essential  in $\M$.
\end{proof}


\subsection{Boundary slopes and homology of cable spaces} A \emph{slope}  $s$ on  a torus $T$ is the isotopy class of a simple closed curve on $T$.
Let $\mathcal {S}(T)$ denote the set of {slopes}  of $T$.
The elements in  $\mathcal {S}(T)$  are represented by elements of
${\QQ}\cup \{ 1/0\}$.  With this in mind we will often refer to a slope by its corresponding numerical value.
A homology class in
$H_1(T,  \ZZ)$ is called \emph{primitive} if it is not a non-trivial integer multiple of another element in $H_1(T,  \ZZ)$.
There is a 2-1 correspondence between primitive  classes in $H_1(T,  \ZZ)$
and elements  in $\mathcal {S}(T)$, where $\alpha, \beta \in H_1(T,  \ZZ)$ give the same slope if and only iff $\alpha=
\pm \beta$.

Next we need the following lemma.

\begin{lemma} \label{glue} Consider  a cable knot complement $\M=\C\cup M_K$  as above.
Let $s\in \mathcal {S}(T_{-})$ be a slope on  $T_{-}$ of $\C$, corresponding to $a/b\in {\QQ}\cup \{ 1/0\}$.
Suppose that in $\C$ we have a properly embedded, connected, essential surface $F$ such that:
\begin{enumerate}
\item The boundary $\partial F$ intersects both of $T_{-}$ and $T_{+}$.

\item  Each component of $\partial F$ on $T_{-}$ has  slope $s$.

\item The total slope of $\partial F\cap T_{+}$ is a boundary slope
in $M_K$.
\end{enumerate}

Then $a/b$ is a boundary slope in $\M$.

\end{lemma}
\begin{proof}  
Suppose that the total slope of $\partial F\cap T_{+}$  corresponds to 
$c/d\in {\QQ}\cup \{ 1/0\}$.
Let $E$ be a essential surface in $M_K$ with  $\partial E$ having  total slope $c/d$.
By passing to the doubles if necessary, we may assume that $E$ and $F$ are orientable.
Let $x$ and $y$  denote the number of components of $\partial F$  and $\partial E$ on $T_{+}$, respectively.
In $\C$ consider a surface $F'$ that is  $y$ copies of  $F$ and in $M_K$ consider a surface $E'$ that is $x$ copies of
$E$. Each of $\partial E'$ and $\partial F'$ has $xy$ components on $T_{+}$.  After isotopy on $T_{+}$ we may assume that $\partial E'=\partial F'$. Now $S=E'\cup F'$ is a properly embedded surface in $\M$.
By assumption, each  component of $S \cut T_{+}$ is essential in the component of $\M \cut T_{+}$ it lies in. Thus  
by Lemma \ref{components}, $S$ is essential in $\M$.
\end{proof}

The following lemma  should be compared with
 \cite[Lemma 2.3]{KS}.

\begin{lemma} \label{bijection} Let $\C$ be a cable space with $\partial C= T_{-}\cup T_{+}$ as above. 
Let $\mathcal {S}_{-}$ and $\mathcal{S}_{+}$ denote the set of slopes on $T_{-}$ and $T_{+}$ respectively.
There is  a bijection
$$\phi: \mathcal {S}_{-}\longrightarrow \mathcal {S}_{+},$$
such that for every $s\in \mathcal {S}_{-}$
there is a connected, essential  properly embedded surface $F\subset \C$, intersecting both components of $\partial \C$,
and such that each component $\partial F\cap T_{-}$ has slope $s$ while each component of $\partial F\cap T_{+}$ has slope $\phi(s)$.

\end{lemma}
\begin{proof}To simplify our notation, throughout this proof, we will use $X:=\C$.
Identify $H_1(\partial X ; \QQ)$  with $H_1(T_{-};  \QQ)\oplus H_1(T_{+};  \QQ)$. 

We claim that the maps 
$i_{\pm}  : H_1(T_{\pm};  \QQ) \longrightarrow H_1(X; \QQ)$, induced by the inclusions of $T_{\pm}$ in $X$ are isomorphisms.
To see that $i_{-}$ is an isomorphism consider the solid torus $X\cup n(T_{p,q})$, and apply the
Mayer--Vietoris long exact sequence to this decomposition. To see that $i_{+}$ is an isomorphism decompose $X$ into a fibered solid torus and an $I$-bundle
$T_{+}\times I$ along a vertical annulus and again apply the Mayer--Vietoris long exact sequence.

The fact that $i_{\pm}$ are isomorphisms implies  the following: Given a primitive class $\alpha_{-} \in H_1(T_{-};  \ZZ) \subset H_1(T_{-}; \QQ)$
there is a unique primitive class $\alpha_{+} \in H_1(T_{+};  \ZZ) \subset H_1(T_{+}; \QQ)$
so that
\begin{equation}
\label{eq-r}
\alpha_{+} = r \ i_{+}^{-1}\circ i_{-}(\alpha_{-}), \  {\text {\rm for \ some }} \   r\in \QQ.
\end{equation}

As discussed above, a  slope $s\in \mathcal {S}_{\pm}$  determines  a primitive class $\alpha_{\pm} \in H_1(T_{\pm};  \ZZ)$ 
up to sign. 
Given  $s\in \mathcal {S}_{-}$, determining a primitive element $\alpha_{-} \in H_1(T_{-},  \ZZ)$  up to sign, 
define $\phi(s)$ to be the slope  in $S_{+}$ that describes the class $\alpha_{+} \in  H_1(T_{+},  \ZZ)$ defined in equation \eqref{eq-r}.
This clearly defines a bijection.

By above discussion, there are relatively prime integers $m,n $  such that 
$$m i_{-}(\alpha_{-})+ n i_{+}(\alpha_{+})=0.$$ Thus the element
$m\alpha_{-} + n\alpha_{+}$ is in the kernel of the map
$i_{-}\oplus i_{+}: H_1(\partial X;  \ZZ) \longrightarrow H_1( X; \ZZ) $. Looking at the homology  long exact sequence for the pair $(X, \partial X)$,
we conclude that there is a class $A \in H_1(X, \partial X; \ZZ)$ such that
$\partial (A)=m\alpha_{-} + n\alpha_{+}$ under the boundary map
$\partial: H_1(X, \partial X; \ZZ) \longrightarrow H_1(\partial X; \ZZ)$.
Now 3-manifold theory assures that there is a 2-sided, embedded, essential  surface $S$ that represents $A$ \cite[Lemma 6.6]{hempel}.
That is $[S]=A$.
By construction, each component of $\partial S\cap T_{-}$ has slope $s$ while each component of $\partial S\cap T_{+}$ has slope $\phi(s)$.

Now $S$ may not be connected. However since $S$ represents $A$, we have $\partial(A)= m\alpha_{-} + n\alpha_{+}\neq 0$. There must be a component $F\subset S$ such that $\partial([F])\neq 0$. We claim that $F$ must intersect both components of $\partial X$.
For, suppose that it doesn't intersect one component of $\partial X$; say $F\cap T_{+}=\emptyset$. Then the class
$\partial([F])\neq 0$ would be a non-zero multiple of $\alpha_{-} \in H_1(\partial X;  \ZZ)$. But this is impossible, since
$i_{-}(\alpha_{-})$ has infinite order in $H_1( X;  \ZZ)$.
Thus $F\cap T_{\pm}\neq \emptyset$. To finish the proof of the lemma, note that since $F\subset S$ and $S$ is essential, each component $\partial F\cap T_{-}$ has slope $s$ while each component of $\partial F\cap T_{+}$ has slope $\phi(s)$.
\end{proof}


\subsection{The proof of Theorem \ref{main:bs}.} We are now ready to  prove Theorem \ref{main:bs}. For part (a) let $K$ be a non-trivial knot and let  $(p, q)$ coprime  integers.
 Suppose that  $a/b$ is a boundary slope of $K$.
Let $\K$ be the cable knot
of $K$. As above we will consider the complement $\M$ obtained by gluing a cable space $\C$  to the complement $M_K$.
We must show that there is an essential surface $S$ in $\M$ such that the total slope of $S\cap \partial \M$ is $q^2a/b$.

On the boundary component $T_{-}=\partial M_{\K}$  consider a pair  $(\mu,  \lambda)$ of meridian and canonical longitude
whose homology classes generate $H_1(T_{-}; \ZZ)$. Let  
$r := \gcd(aq^2, b)$ and let
$x= a q^2 /r$ and $y=b/r$.
Consider
a simple closed curve $\gamma$ whose numerical slope is $x/y$. That is in $H_1(T_{-}; \ZZ)$ we have
$$[\gamma]= x\mu+ y\lambda.$$

 By Lemma \ref{bijection}, applied for  $s=x/y$ we have  a slope $\phi(s)$ 
on the other component  $T_{+}$ of $\C$ such that the following is true: There is an essential, connected,  properly embedded surface $F\subset \C$, with  $\partial F$ intersecting both components of $\partial \C$,
and such that each component $\partial F\cap T_{-}$ has slope $s$ while each component of $\partial F\cap T_{+}$ has slope $\phi(s)$.

Since $\C$ is a Seifert fibered space, up to isotopy,  essential surfaces are either vertical or horizontal with respect to the Seifert fibration \cite{Ha}. Since the base space of $\C$ is an annulus, the only vertical surfaces in $\C$ are annuli.
In fact, essential surfaces in cable spaces are classified by Gordon and Litherland in \cite[Lemma 3.1]{GL}. From that lemma and its proof, it follows that if $F$ is a vertical annulus in $\C$, then every component of $\partial F\cap T_{-}$  has slope $pq$, while every component of 
$\partial F\cap T_{+}$ has slope $p/q$.  In particular, we have $a/b=p/q$ and our hypothesis 
implies that for this to happen  $p/q$ must be a boundary slope of $K$. On the other hand $x/y=pq$ which, by Lemma \ref{cannulus},
is always a boundary slope of $\M$. Thus the desired conclusion holds in this case.

Now suppose that the surface $F$ is horizontal with respect to the Seifert fibration of $\C$.
On $T_{-}$ a regular fiber of the fibration can be identified with the knot $H:=\K$.
The only singular fiber of the fibration  (that has multiplicity $\abs{q}$) may be identified with the knot $K$ on $T_{+}$. If $N$ is the number of times that $F$ intersects the regular fibers of $\C$, then since the base of the fibration is an annulus we have
$$0-\chi(F)/N= 1-1/\abs{q}.$$
It follows that  $\chi(F)=n' (1-\abs{q})$, where
 $N=n'q$. 
Since the intersection number of $\mu$ and $H$ is $1$ we conclude that on $\partial T_{-}$,  the meridian curve $\mu$ is a cross section of the fibration and in $H_1(\partial \M)$ we have
$[\gamma]=nq  \mu +mH$, where  $n=n'/r, m\in \ZZ$ and we have $(nq, m)=1$.
On the other hand we must have  $H= pq \mu + \lambda$.
Hence we obtain that each component of   $$[\partial F]=r(nq +mq) \mu +rm\lambda, \ \ \  {\rm thus} \ \ \  s=x/y= (nq +mpq)/m.$$

In $\C$ we have a horizontal planar surface that intersects $T_{+}$ in a single meridian, call it $\mu'$,  and it intersects
$T_{-}$ in $q$ copies of $\mu$. This surface may be taken to be the image of a meridian disk of a neighborhood of $K$ in $\C$.
By choosing appropriate orientations of $\mu, \mu'$, and by the proof of Lemma \ref{bijection},  we may assume that $\phi(\mu)=\mu'$ and that in $H_1(\C ; \ZZ)$ we have $\mu'=q m$.

In $H_1(T_{+}; \ZZ)$,  the fiber $H$ corresponds to the slope $p/q$
and each component of  $T_{+}\cap \partial F$ has the form
$n  \mu' + mH$. Thus,  
as also stated in \cite[Lemma 3.1]{GL}, $T_{+}\cap \partial F$ 
 has total slope $(n+mp)/(mq)$. The number of components
$T_{+}\cap \partial F$ is $t= \gcd(n+mp, mq)$. Since in $\QQ$ we have that 
 $(n+mp)/(mq)=a/b$ we may take $a':=a/w=n+mp/t$ and $b':=b/w=mq/t$, where $w=\gcd(a, b)$.
Thus we may assume that  $\phi(s)=a'/b'$. 
By assumption $a/b$  is a boundary slope in $M_K$. Thus, Lemma \ref{glue} applies
to conclude that $s=q^2a/b$ is a boundary slope in $\M$. This finishes the proof of part (a) of
 Theorem \ref{main:bs}.
 
For a non-trivial knot $K$, part (b) follows at once from part (a).
To finish the proof of part (b) assume that $K$ is the trivial knot and let  $(p, q)$ coprime  integers.
Now $\K$ is the $(p, q)$ torus knot. The only boundary slope of $K$ is  $a/b=0/1$ and the boundary slopes of
$\K$ are $0$ and $pq$. Thus $q^2 bs_K \cup \{pq\} =\{0, pq \}= bs_{K_{p,q}}$. \qed
\medskip

We close the section with the following corollary that will be useful to us in subsequent sections.

\begin{corollary}\label{integers} Let $\M=\C\cup M_K$ be the complement of a cable knot, where $\abs{q}>1$.
Let $F\subset \C$ be a properly embedded essential surface, that is not an annulus, and such that each   component of 
 $\partial F\cap T_{+}$  has integral slope $a$. 
Suppose   that there is a connected essential surface
$S'\subset M_K$ such that each component of $\partial S'$ has slope $a$. Then the followings hold.
\begin{enumerate}
\item 
 $\partial F\cap T_{+}$  has $\abs{q}$ components and  $\partial F\cap T_{-}$ has a single component
 of slope $q^2a\in \ZZ$.

\item There is a connected essential surface $S\subset \M$, such that each component of $\partial S$ has slope
$q^2a$ and 
$$\chi(S)=    \abs{q} \chi(S') + \abs{ \partial S'}(1-\abs{q})\abs{p-aq},$$
\noindent where $\abs{ \partial S'}$ denotes the number of boundary components of $S'$.
Furthermore, we have $\abs{ \partial S}= \abs{ \partial S'}$.

\end{enumerate}

\end{corollary}
\begin{proof} Let $F$ be as in the statement above.
By the proof of Theorem \ref{main:bs}, since $F$ is not an annulus, $ F$ may be isotopied to be horizontal with respect to the Seifert fibration of $\C$.
Furthermore, $\partial F\cap T_{+}$
has total  slope  $(n+mp)/(mq)$, while 
$\partial F\cap T_{-}$ has total slope  $ (nq +mpq)/m$,
for some coprime  integers $m,n$.
The number of components of $\partial F\cap T_{+}$ is $t= \gcd(n+mp, mq)$,
and each has slope $b/c$ where $b=(n+mp)/t$ and $c=mq/t$.
Since  $b/c\in \ZZ$ and $\gcd(m,n)=1$, it follows that $m=\pm1$ and $t=\abs{q}$.
Hence 
 $n=m(qa-p)$, where we will have $m=1$ or $m=-1$ according to whether $qa\geq p$
or $qa\leq p$. Furthermore $n>0$, $\chi(F)=n(1-\abs{q})$. 
The rest of the claims  in part (1) follow.

Now we prove part (2): By the proof of Theorem \ref{main:bs}, an essential  surface $S$ realizing the boundary slope $q^2a$ for $\K$ 
is as in the proof of Lemma \ref{glue}. Since   $\partial F\cap T_{+}$  has $\abs{q}$ components, $S$ is constructed by 
gluing  $\abs{q}$ copies of $S'$ with $ \abs{ \partial S'}$ copies of $F$. Hence we have
$$\chi(S)=\chi(M_K\cap S)+\chi(\C \cap S)=\abs{q}  \chi(S')+\abs{\partial S'} \, \abs{aq-p} (1- \abs{q} ).$$
The last equation follows from the fact that
$\chi(\C \cap S)=\abs{\partial S'}  \chi(F)$ and the above discussion on
  $\chi(F)$.
\end{proof}

\section{Cables of Knots with period at most two}
\label{cjp} In this section we study the behavior of  the Jones slopes of  knots under the operation of cabling. The main result is Theorem
\ref{thm-quasi} that relates the Jones slopes of knots of period at most two to the Jones slopes of their cables. We apply this theorem to prove the Strong Slope Conjecture for iterated cables of adequate knots and  iterated torus knots and the Slope Conjecture for cables of all the non-alternating knots up to nine crossings that have period two.

 \subsection{The colored Jones polynomial} To define the colored Jones polynomial, we first recall the definition of the Chebyshev polynomials of the second kind. For $n \ge 0$, the polynomial $S_n(x)$ is defined recursively as follows:
\begin{equation}
\label{chev}
S_{n+2}(x)=xS_{n+1}(x)-S_{n}(x), \quad S_1(x)=x, \quad S_0(x)=1.
\end{equation}

Let $D$ be a diagram of a knot $K$. For an integer $m>0$, let $D^m$ denote the diagram obtained from $D$ by taking $m$ parallels copies of $K$. This is the $m$-cable of $D$ using the blackboard framing; if $m=1$ then $D^1=D$. Let $\la D^m \ra$ denote the Kauffman bracket of $D^m$: this is a Laurent polynomial over the integers in a variable $v^{-1/4}$ normalized so that $\la \text{unknot} \ra = -(v^{1/2}+v^{-1/2})$. Let $c=c(D)=c_+ + c_-$ denote the crossing number and $w=w(D)=c_+ - c_-$ denote the writhe of $D$. 

For $n>0$, we define 
$$J_K(n):=  ( (-1)^{n-1} v^{(n^2-1)/4} )^w (-1)^{n-1}  \la S_{n-1}(D)\ra$$
where $S_{n-1}(D)$ is a linear combination of blackboard cablings of $D$, obtained via equation \eqref{chev}, and the notation $\la S_{n-1}(D) \ra$ means extend the Kauffman bracket linearly. That is, for diagrams $D_1$ and $D_2$ and scalars $a_1$ and $a_2$, $\la a_1 D_1 + a_2 D_2 \ra = a_1 \la D_1 \ra + a_2 \la D_2 \ra$.

For a Laurent polynomial $f(v) \in \BC[v^{\pm 1/4}]$, let $d_+[f]$ and $d_-[f]$ be respectively the maximal and minimal degree of $f$ in $v$.

\begin{define}
A \emph{quasi-polynomial} is a function $$f:\BN \to \BC, \quad f(n)=\sum_{i=0}^d c_i(n)n^i$$
for some $d \in \BN$, where $c_i(n)$ is a periodic function with integral period for $i=1, \cdots, d$. If $c_d(n)$ is not identically zero, then the degree of $f(n)$ is $d$. 

The \emph{period} $\pi$ of a quasi-polynomial $f(n)$ as above is the common period of $c_i(n)$.
\end{define}

Garoufalidis \cite{ga-quasi} showed that for any knot $K\subset S^3$
the degrees $ d_+[J_{K}(n)] $ and $ d_-[J_{K}(n)] $ are quadratic quasi-polynomials.
The common period  of $ d_+[J_{K}(n)] $ and $ d_-[J_{K}(n)]$ is called \emph{the period of $K$}, denoted  by $\pi(K)$. 
\medskip


\subsection{Cables of knots of period at most 2} 

\label{period_2}

In this subsection we will study knots with period at most two. Examples of such knots include all the adequate knots and the torus knots. We show that, under a mild hypothesis satisfied by all the known examples, the property of having period at most two is preserved under cabling (Proposition \ref{quasi}). As a  result, if a knot $K \subset S^3$ satisfies the Slope Conjecture and $\pi(K)\leq 2$, then all but at most two cables of $K$ also satisfy the conjecture.

\begin{proposition}
\label{quasi}
Let $K$ be a knot such that for $n \gg 0$ we have
$$
d_+[J_K(n)]=a(n)n^2+b(n)n+d(n)
$$
is a quadratic quasi-polynomial of period $\le 2$, with $b(n) \le 0$. 

Suppose $\frac{p}{q} \notin \{4a(n)\}$. Then for $n \gg 0$ we have
$$d_+[J_{\K}(n)]=A(n)n^2+B(n)n+D(n)$$
is a quadratic quasi-polynomial of period $\le 2$, with $\{A(n)\} \subset \big( \{q^2a(n)\} \cup \{pq/4\} \big)$ and $B(n) \le 0$.
\end{proposition}

\begin{proof} 

Since $d_+[J_K(n)]$ is a quadratic quasi-polynomial of period $\le 2$, for $n \gg 0$ we can write

$$d_+[J_K(n)]=\begin{cases} 
a_0 n^2+b_0n+d_0 & \mbox{if } n \mbox{ is even},\\
a_1 n^2+b_1n+d_1 & \mbox{if } n \mbox{ is odd}.\end{cases}$$

 Recall that $K_{p,q}$ is the $(p,q)$-cable of a knot $K$, where $p,q$ are coprime integers and $|q|>1$. It is known that $K_{-p,-q}=\text{r}K_{p,q}$, where $\text{r}K_{p,q}$ denotes $\K$ with the opposite orientation, and that the colored Jones polynomial of a \textit{knot} is independent of the orientation of the knot. Hence, without loss of generality, we will assume that $q > 1$.

 For $n>0$, let $\CS_n$ be the set of all $k$ such that 
 $$
 |k| \le (n-1)/2 \quad \text{and} \quad k \in \begin{cases} \BZ &\mbox{if } n \text{ is odd}, \\ 
 \BZ+\frac{1}{2} & \mbox{if } n \text{ is even}. \end{cases}
 $$
By \cite{Ve}, for $n>0$ we have
\begin{equation}
\label{cables}
J_{K_{p,q}}(n)= v^{pq(n^2-1)/4} \sum_{k \in \CS_n} v^{-pk(qk+1)} J_K(2qk+1),
\end{equation} 
where it is understood that $J_K(-m)=-J_K(m)$.

In the above formula, there is a sum. Under the assumption of the proposition, we will show that there is a unique term of the sum whose highest degree is strictly greater than those of the other terms. This implies that the highest degree of the sum is exactly equal to the highest degree of that unique term.

Let $\CS^+_n=\{k \in \CS_n \mid k \ge 0\}$ and $\CS^-_n=\{k \in \CS_n \mid k \le -\frac{1}{2}\}$.  For $k \in \CS_n$ let $$f(k)=d_+[v^{-pk(qk+1)} J_K(2qk+1)].$$
Let $g_i^{\pm}(x)$, for $i \in \{0,1\}$, be the quadratic real polynomials defined by
$$
g_i^{\pm}(x) = ( -pq + 4q^2 a_i ) x^2 + ( -p + 4q a_i \pm 2q b_i ) x +  a_i \pm b_i + d_i.
$$
For $k \in \CS_n$ we have 
$f(k) = -pk(qk+1)+d_+[J_K(|2qk+1|)]$, which gives
\begin{eqnarray*}
 f(k)=
\begin{cases} 
g_0^+(k) &\mbox{if } k \in \CS^+_n \mbox{ and } 2qk+1 \mbox{ is even}, \\
g_0^-(k) &\mbox{if } k \in \CS^-_n \mbox{ and } 2qk+1 \mbox{ is even},\\
g_1^+(k) &\mbox{if } k \in \CS^+_n \mbox{ and } 2qk+1 \mbox{ is odd}, \\
g_1^-(k) &\mbox{if } k \in \CS^-_n \mbox{ and } 2qk+1 \mbox{ is odd}. 
\end{cases}
\end{eqnarray*}

\smallskip

\underline{\textit{Case 1.}}\   Suppose that $q$ is even.
For $k \in \CS_n$, since $2qk+1$ is odd we have $$f(k)= \begin{cases} g_1^+(k) &\mbox{if } k \in \CS^+_n, \\
g_1^-(k) &\mbox{if } k \in \CS^-_n.  \end{cases}$$
\vskip 0.1in

\textit{Subcase 1.1.} Assume  that  $p/q <4a_1$.

Since $-pq + 4q^2 a_1>0$, the quadratic polynomial $g_1^+(x)$ is concave up. Hence, for $n \gg 0$, $g_1^+(k)$ is maximized on $\CS^+_n$ at $k=(n-1)/2$.  Similarly, $g_1^-(k)$ is maximized on $\CS^-_n$ at $k=(1-n)/2$. Note that $$g_1^+ \big( (n-1)/2 \big)-g_1^-  \big( (1-n)/2 \big)=(-p+4qa_1)(n-1)+2b_1>0$$
for $n \gg 0$. Hence $f(k)$ is maximized on the set $\CS_n$ at $k=(n-1)/2$. Since $f \big( (n-1)/2 \big)=g^+_1 \big( (n-1)/2 \big)$, equation \eqref{cables} then implies that  
\begin{eqnarray*}
d_+[J_{K_{p,q}}(n)] &=& pq(n^2-1)/4+g_1^+ \big((n-1)/2 \big)\\
&=& q^2a_1n^2+ \big(qb_1+(q-1)(p-4qa_1)/2 \big)n\\
&& + \, a_1(q-1)^2-(b_1+p/2)(q-1)+d_1
\end{eqnarray*}
for $n \gg 0$. Since we assumed that $q>1$, we have that $B(n) = qb_1+(q-1)(p-4qa_1)/2<0$, and the conclusion follows in this case.
\vskip 0.1in

\textit{Subcase 1.2.} 
Assume that $p/q >4a_1$. 

Since $-pq + 4q^2 a_1<0$, the quadratic polynomial $g_1^+(x)$ is concave down and attains its maximum at 
$$x=x_0:=- \left(\frac{1}{2q}+\frac{b_1}{-p+4qa_1} \right).$$ Since $b_1 \le 0$, we have $x_0 < 0$. This implies that $g_1^+(x)$ is a strictly decreasing function on $[0,\infty)$. Similarly, $g_1^-(x)$ is a strictly increasing function on $(-\infty,-\frac{1}{2}]$.  

First suppose $n$ is even. Then $k \in \BZ+\frac{1}{2}$. In this subcase, $g_1^+(k)$ is maximized on $\CS^+_n$ at $k=\frac{1}{2}$ and $g_1^-(k)$ is maximized on $\CS^-_n$ at $k=-\frac{1}{2}$. Note that $g_1^+(\frac{1}{2})-g_1^-(-\frac{1}{2})=(-p+4qa_1)+2b_1<0$. Hence $f(k)$ is maximized on $\CS_n$ at $k=-1/2$. Since $f(-\frac{1}{2})=g_1^-(-\frac{1}{2})$, equation \eqref{cables} then implies that
$$d_+[J_{K_{p,q}}(n)]=pq(n^2-1)/4+g^-_1(-1/2)$$
for even $n \gg 0$. Similarly, for odd $n \gg 0$ we obtain
$$d_+[J_{K_{p,q}}(n)]=pq(n^2-1)/4+g^+_1(0).$$
Note that $B(n) = 0$ in this case.
\vskip 0.2in

\underline{\textit{Case 2.}} \ Suppose that $q$ is odd. 
As in Case 1 we have the following.
\vskip 0.1in

\textit{Subcase 2.1.}  Suppose that $n$ is even. 

For $k \in \CS_n$, we have $k \in \BZ+\frac{1}{2}$ and $2qk+1$ is even. Hence $f(k)= \begin{cases} g_0^+(k) &\mbox{if } k \in \CS^+_n, \\
g_0^-(k) &\mbox{if } k \in \CS^-_n.  \end{cases}$

If $p/q <4a_0$ then $f(k)$ is maximized on $\CS_n$ at $k=(n-1)/2$. Hence
\begin{eqnarray*}
d_+[J_{K_{p,q}}(n)] &=& pq(n^2-1)/4+g_0^+ \big( (n-1)/2 \big)\\
&=& q^2a_0n^2+ \big(qb_0+(q-1)(p-4qa_0)/2 \big)n\\
&& + \, a_0(q-1)^2-(b_0+p/2)(q-1)+d_0.
\end{eqnarray*}
In this case we have $B(n) = qb_0+(q-1)(p-4qa_0)/2<0$.

If $p/q >4a_0$ then $f(k)$ is maximized on $\CS_n$ at $k=-1/2$. Hence $$d_+[J_{K_{p,q}}(n)]=pq(n^2-1)/4+g^-_0(-1/2).$$
Note that $B(n) = 0$ in this case.

\vskip 0.1in

\textit{Subcase 2.1.}  Suppose that $n$ is odd. 

For $k \in \CS_n$, we have $k \in \BZ$ and $2qk+1$ is odd. Hence $f(k)= \begin{cases} g_1^+(k) &\mbox{if } k \in \CS^+_n, \\
g_1^-(k) &\mbox{if } k \in \CS^-_n.  \end{cases}$

If $p/q <4a_1$ then $f(k)$ is maximized on $\CS_n$ at $k=(n-1)/2$. Hence
\begin{eqnarray*}
d_+[J_{K_{p,q}}(n)] &=& pq(n^2-1)/4+g_1^+ \big( (n-1)/2 \big)\\
&=& q^2a_1n^2+ \big(qb_1+(q-1)(p-4qa_1)/2 \big)n\\
&& + \, a_1(q-1)^2-(b_1+p/2)(q-1)+d_1.
\end{eqnarray*}
In this case we have $B(n) = qb_1+(q-1)(p-4qa_1)/2 < 0$.

If $p/q >4a_1$ then $f(k)$ is maximized on $\CS_n$ at $k=0$. Hence $$d_+[J_{K_{p,q}}(n)]=pq(n^2-1)/4+g^+_1(0).$$
Note that $B(n) = 0$ in this case.

This completes the proof of Proposition \ref{quasi}.
\end{proof}

\begin{remark} (1) Proposition \ref{quasi} generalizes \cite[Lemma 2.2]{RZ}, \cite[Lemma 2.2]{Ru},  \cite[Lemma 3.1]{Tr-AJ-8} and \cite[Lemma 3.2]{Tr}.

(2) When $\pi(K)$ is greater than 2 then determining the highest degree of
 $J_{K_{p,q}}(n)$ in Equation \eqref{cables} becomes harder as there might be more opportunities for cancellation
between terms. This case will be discussed in Section \ref{constant a(n)}.
\end{remark}
 
 Proposition \ref{quasi} and Theorem \ref{main:bs} imply the following.
 
 \begin{theorem}
 \label{thm-quasi}
 Let $K$ be a knot such that for $n \gg 0$ we have
$$
d_+[J_K(n)]=a(n)n^2+b(n)n+d(n)
$$
is a quadratic quasi-polynomial of period $\le 2$, with $b(n) \le 0$. 

Suppose $\frac{p}{q} \notin js_K$. Then for $n \gg 0$ we have
$$d_+[J_{\K}(n)]=A(n)n^2+B(n)n+D(n)$$
is a quadratic quasi-polynomial of period $\le 2$, with $B(n) \le 0$. 
Moreover,  if $js_K \subset bs_K$ we have $js_{\K} \subset bs_{\K}$.

Similarly, let $K$ be a knot such that for $n \gg 0$ we have
$$
d_-[J_K(n)]=a^*(n)n^2+b^*(n)n+d^*(n)
$$
is a quadratic quasi-polynomial of period $\le 2$, with $b^*(n) \ge 0$. 
Suppose $\frac{p}{q} \notin js^*_K$.  
Then for $n \gg 0$ we have
$$d_-[J_{\K}(n)]=A^*(n)n^2+B^*(n)n+D^*(n)$$
is a quadratic quasi-polynomial of period $\le 2$, with $B^*(n) \ge 0$. 
Moreover,  if $js^*_K \subset bs_K$ we have $js^*_{\K} \subset bs_{\K}$.
 \end{theorem}
\begin{proof} The first part of the theorem follows immediately by Proposition  \ref{quasi} and Theorem \ref{main:bs}.
To obtain the second part recall that if ${K^{*}}$ denotes the mirror image of $K$ then
$J_{K^{*}}(n)$ is obtained from $J_K(n)$ by replacing the variable $v$ with $v^{-1}$.
Now the result will follow by applying Proposition \ref{quasi} and Theorem \ref{main:bs} to $K^{*}$.
\end{proof}
 
\smallskip

\subsection{Strong Slope Conjecture for iterated cables of adequate knots} 
Let $D$ be a link diagram, and $x$ a crossing of $D$.  Associated to
$D$ and $x$ are two link diagrams, each with one fewer crossing than
$D$, called the \emph{$A$--resolution} and \emph{$B$--resolution} of
the crossing. See Figure \ref{resolutions}.

A Kauffman state $\sigma$ is a choice of $A$--resolution or
$B$--resolution at each crossing of $D$. Corresponding to every state
$\sigma$ is a crossing--free diagram $s_\sigma$: this is a collection
of circles in the projection plane. We can encode the choices that lead
to the state $\sigma$ in a graph $G_\sigma$, as follows. The vertices
of $G_\sigma$ are in $1-1$ correspondence with the state circles of
$s_\sigma$. Every crossing $x$ of $D$ corresponds to a pair of arcs
that belong to circles of $s_\sigma$; this crossing gives rise to an
edge in $G_\sigma$ whose endpoints are the state circles containing
those arcs.

Every Kauffman state $\sigma$ also gives rise to a surface $S_\sigma$,
as follows. Each state circle of $\sigma$ bounds a disk in $S^3$. This
collection of disks can be disjointly embedded in the ball below the
projection plane. At each crossing of $D$, we connect the pair of
neighboring disks by a half-twisted band to construct a surface
$S_\sigma \subset S^3$ whose boundary is $K$. See Figure
\ref{statesurface} for an example where $\sigma$ is the all--$B$
state.

\begin{figure}
  \includegraphics{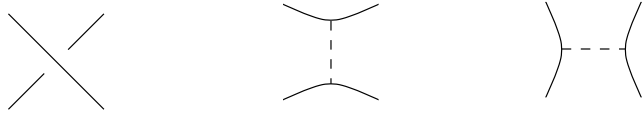}
  
  \caption{From left to right: A crossing, the $A$-resolution and  the
 the $B$-resolution.}
    
  \label{resolutions}
\end{figure}

\begin{figure}
  \includegraphics{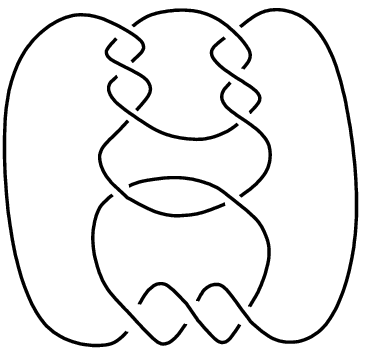} \hspace{.2in}
  \includegraphics{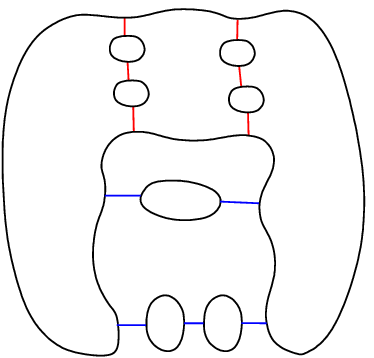} \hspace{.2in}
  \includegraphics{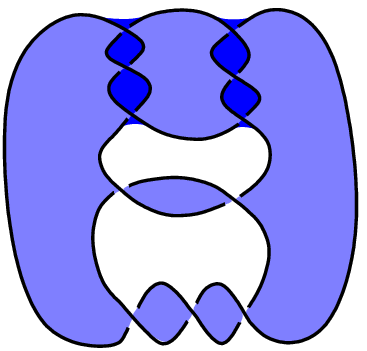}
  \caption{From left to right: An example link diagram, the graph
    $\G_{\sigma}$ corresponding to state surface
    $S_{\sigma}$. \thanks{Figure borrowed from \cite{fkp:qsf}.}}
  \label{statesurface}
\end{figure}

\begin{define}\label{defi:adequate}
A link diagram $D$ is called \emph{$A$--adequate} if the state graph
$\GA$ corresponding to the all--$A$ state contains no 1--edge loops.
Similarly, $D$ is called \emph{$B$--adequate} if the all--$B$ graph $\GB$
contains no 1--edge loops.  A link diagram 
is \emph{adequate} if it is
both $A$-- and $B$--adequate.  A link that admits an adequate diagram
is also called \emph{adequate}.
\end{define}

The number of negative crossings  $c_-$  of an $A$--adequate knot diagram is a knot invariant.
Similarly, the number of positive crossings $c_+$ of  a  $B$-adequate knot diagram is a knot invariant.
Let $v_A$ (resp. $v_B$) be the number of state circles in the all--$A$ (resp. all--$B$) state of the knot diagram $D$.

The following summarizes \cite[Lemma 5.4]{Li}, \cite[Proposition 2.1]{Le06} and \cite[Theorem 3.1]{KL}.

\begin{lemma} 
\label{degLe}
Let $D$ be a diagram of a knot $K$.

(1) We have
$$2 \, d_-[J_{K}(n)] \ge  -c_- n^2 + (c -v_A) n  + v_A - c_+.$$
Equality holds for all $n \ge 1$ if $D$ is $A$--adequate. Moreover, if equality holds for some $n \ge 3$ then $D$ is $A$--adequate.

(2) We have
$$
2 \, d_+[J_{K}(n)] \le c_+ n^2 + (v_B - c) n + c_- - v_B.
$$
Equality holds for all $n \ge 1$ if $D$ is $B$--adequate. Moreover, if equality holds for some $n \ge 3$ then $D$ is $B$--adequate.
\end{lemma}

The following theorem, which implies that the Slope Conjecture is true for adequate knots, summarizes results proved in   \cite{FKP, FKP-guts}.

\begin{theorem}\label{thm:essential} 
Let $D$ be an $A$--adequate diagram of a knot $K$. Then the state
surface $S_A$ is essential  in the knot
complement $M_K$, and it  has boundary slope $-2c_-$. Furthermore, we have
$$-2c_- = \lim_{n \to \infty} 4\, n^{-2}d_-[J_K(n)].$$

Similarly, if $D$ is a $B$--adequate diagram of a
knot $K$, then $S_B$ is essential in the knot 
complement $M_K$, and it  has boundary slope $-2c_+$.
 Furthermore, we have
$$2c_+ =  \lim_{n \to \infty} 4\, n^{-2}d_+[J_K(n)].$$
\end{theorem}

By Lemma 
\ref{degLe}, the highest degree of the colored Jones polynomial of a $B$-adequate knot is 
an actual quadratic polynomial in $n$. That is the period is one.  The following lemma shows that
the term $b\leq 0$ and thus non-trivial $B$-adequate knots satisfy the hypothesis of Proposition \ref{quasi}.

\begin{lemma}
\label{v_B-c}
Suppose $D$ is a diagram of a non-trivial knot $K$. Then $v_B \le c$. Furthermore, 
if  $v_B = c$ then $K$ is a torus knot.
\end{lemma}

\begin{proof} Let $D$ be a diagram of a non-trivial knot and let $S_B$ be the all-$B$ state surface obtained from $D$.
Recall that $S_B$ is a surface with a single boundary component obtained by starting with $v_B$ disks 
and attaching 
a half-twisted band for each crossing of $D$. Thus the Euler characteristic of $S_B$ is
$\chi(S_B)=v_B-c$. Since $D$ represents a non-trivial knot, we have
$\chi(S_B)\le 0$ and thus $v_B \le c$.

If $\chi(S_B)= v_B-c=0$, then  (since $\partial S_B$ has one component) $S_B$ must be a Mobius band. This implies
that $D$ is the standard closed 2-braid diagram of   a $(2, q)$-torus  knot.
\end{proof}

The above discussion shows that the first part of Theorem \ref{thm-quasi} applies to non-trivial $B$--adequate knots. Similarly the second part of the theorem applies to non-trivial $A$--adequate knots.
We are now ready to  prove  the following theorem that implies Theorem \ref{main:js} stated in the introduction.

\begin{theorem}\label{slopeadequate}
Let $K$ be a $B$-adequate knot and $K':=K_{(p_1,q_1),(p_2,q_2), \cdots, (p_r,q_r)}$ an iterated cable knot of $K$. 
Then we have $js_{K'} \subset bs_{K'}$.
Furthermore, for $n \gg 0$, 
$$
d_+[J_{K'}(n)]=An^2+Bn+D(n),
$$
is a quadratic quasi-polynomial of period $\le 2$, with $4A \in \BZ$,  $2B \in \BZ$ and $B \le 0$, and there is an essential surface $S'$ in the complement of $K'$  
with boundary slope $4A$ and such that $\chi(S')=2B$. In particular $K'$ satisfies the Strong Slope Conjecture.
\end{theorem}

\begin{proof}
Suppose $K$ is a $B$--adequate knot. First, we prove that the conclusion of Theorem \ref{slopeadequate} holds true for the cable knot $K_{p_1,q_1}$. We will distinguish two cases according to whether $K$ is a non-trivial knot or not.

\smallskip

\underline{\textit{Case 1.}}\ Suppose that $K$ is non-trivial. 
Then, by Lemmas \ref{degLe} and \ref{v_B-c} we have
$$
d_+[J_K(n)]=an^2+bn+d
$$
for $n>0$, where $a=c_+/2$, $b = (v_B - c)/2 \le 0$ and $d = (c_- - v_B)/2$. 
By Theorem \ref{thm:essential}, $4a=2c_+$ is a boundary slope of $K$. 
Furthermore, an essential surface that realizes this boundary slope
is the state surface $S_B$. Since $S_B$ is constructed by joining $v_B$ disks with $c$ bands we have $\abs{\bdy S_B}=1$ and $\chi(S_B)=v_B-c=2b$. Thus the conclusion is true in this case.

Now we consider a  cable  $K_{p_1,q_1}$ of $K$.
Since $\abs{q_1}>1$, and the Jones slopes of $K$ are integers,
we have  $\frac{p_1}{q_1} \notin js_K$. 
Theorem \ref{thm-quasi} and the proof of Proposition \ref{quasi} then imply that
$$
d_+[J_{K_{p_1,q_1}}(n)]=A_1 n^2+B_1n+D_1(n)
$$
is a quadratic quasi-polynomial of period $\le 2$, with $4A_1 \in \BZ, 2B_1 \in \BZ$ and $B_1 \le 0$.
Moreover, since $js_K \subset bs_K$ we have $js_{K_{p_1,q_1}} \subset bs_{K_{p_1,q_1}}$.
Furthermore, the proof of Proposition \ref{quasi} show that one of the following is true:
\begin{enumerate}
\item We have $4A_1=pq$ and $B_1=0$.
\item We have $4A_1=4q^2a$ and $2B_1=2 \abs{q} b+ (1- \abs{q}) \abs{4aq-p}$.

\end{enumerate}
 In case (1), the surface $S$ with boundary slope $pq$ is the cabling annulus; thus $\chi(S)=0=2B_1$.
In case (2), an essential  surface $S$ realizing the boundary slope $4A_1=4q^2a=2q^2c_{+}$
is obtained by Theorem \ref{main:bs}.  By Corollary \ref{integers} we have $\abs{\bdy S} = \abs{\bdy S_B} = 1$ and 
$\chi(S)= \abs{q} \chi(S_B)+ (1- \abs{q}) \abs{4aq-p}=2B_1$. Thus the conclusion follows for  $K_{p_1,q_1}$.

\smallskip

\underline{\textit{Case 2.}}\
Suppose that $K$ is the trivial knot.  Then $K_{p_1,q_1}$ is the $(p_1,q_1)$-torus knot. Note that $0$ and $p_1q_1$ are boundary slopes of $K_{p_1,q_1}$; realized by a Seifert surface and an annulus respectively.

 For $n>0$, by \cite{Mo} (or by equation \eqref{cables}) we have
$$J_{K_{p_1,q_1}}(n)= v^{p_1q_1(n^2-1)/4} \sum_{k \in \CS_n} v^{-p_1k(q_1k+1)} \frac{v^{(2qk+1)/2}-v^{-(2qk+1)/2}}{v^{1/2}-v^{-1/2}}.$$
By \cite[Section 4.8]{Ga-slope}, \cite[Lemma 1.4]{Tr-strongAJ} and \cite[Lemma 2.1]{RZ}, we have the following. If $p_1>0$ and $q_1>0$ then
$$d_+[J_{K_{p_1,q_1}}(n)]= \big( p_1q_1n^2+d(n) \big)/4$$
where $d(n)=-p_1q_1-\frac{1}{2} \big( 1+(-1)^{n}\big)(p_1-2)(q_1-2)$ is a periodic sequence of period $\le 2$. In this case we have $A_1(n) = p_1q_1/4$ and $B_1(n) = 0$.

If $p_1<0<q_1$ then 
$$d_+[J_{K_{p_1,q_1}}(n)]= \big( (p_1q_1-p_1+q_1)n-(p_1q_1-p_1+q_1) \big)/2.$$
In this case we have $A_1(n) = 0$ and $B_1(n) = (p_1q_1-p_1+q_1)/2$. Note that $p_1q_1-p_1+q_1=1+(p_1+1)(q_1-1) \le 0$.

In both cases we have that $$
d_+[J_{K_{p_1,q_1}}(n)]=A_1n^2+B_1n+D_1(n)
$$
is a quadratic quasi-polynomial of period $\le 2$, with $4A_1 \in \BZ, 2B_1 \in \BZ$ and $B_1 \le 0$. Moreover, we have $$js_{K_{p_1,q_1}} \subset bs_{K_{p_1,q_1}}.$$

For $p_1>0$ and $q_1 >0$, an essential surface with slope $p_1q_1$ for $K_{p_1,q_1}$ is the cabling annulus $A$; thus $\chi(A_1)=0=2B_1$.
For $p_1<0<q_1$, the genus of $K_{p_1,q_1}$ is $g= -(q_1-1)(p_1+1)/2$.
A Seifert surface $S$ of minimal genus for $K_{p_1,q_1}$ has boundary slope $0$ and 
$\chi(S)=1-2g=p_1q_1-p_1+q_1=2B_1$. This proves the desired conclusion for the torus knot $K_{p_1,q_1}$.

We have proved that the conclusion of Theorem \ref{slopeadequate} holds true for the cable knot $K_{p_1,q_1}$. Now, applying the arguments in Case 1 repeatedly will finish the proof of the theorem for iterated cables.
\end{proof}

 \subsection{ Proof of Theorem \ref{main:js}}   Part (1) is immediate from Theorem \ref{slopeadequate}.
Suppose now that $K$ is $A$-adequate. Then the mirror image $K^{*}$ is $B$-adequate.
Furthermore, $J_{K^{*}}(n)$ is obtained from $J_{K}(n)$ by replacing $v$ with $v^{-1}$. Thus in this case the result will follow
by applying  Theorem \ref{slopeadequate} to $K^{*}$.
\smallskip

\subsection{Low crossing knots with period two}
Theorem \ref{thm-quasi} applies to several non-alternating knots with 8 and 9 crossings.  The Jones slopes of all non-alternating prime knots with up to nine crossings were calculated by Garoufalidis in Section 4 of \cite{Ga-slope}.
According to \cite{Ga-slope} the knots of 
 period two are the ones shown in Table 1. 
The remaining knots, which are $8_{20}, 9_{43}, 9_{44}$, have periods 3 and they will be treated in Section \ref{constant a(n)}.

\begin{table}[]
\begin{center}
\begin{tabular}{|c|c|c|c|c|c|c|c|c|}
\hline
$K$& $js_K \cup js^*_K $ & $b(n)$ & $b^*(n)$ \\
\hline
$8_{19}$ &$ \{12,0\}$ &  $0$ & $5/2$ \\
\hline
$8_{21}$ & $\{1,-12\}$ & $-1$ & $3/2$ \\
\hline
$9_{42}$  & $\{6,-8\}$ & $-1/2$ & $5/2$ \\
\hline
$9_{45}$  & $ \{1,-14\}$ & $-1$ & $2$ \\
\hline
$9_{46}$  & $ \{2,-12\}$ & $-1/2$ & $5/2$ \\
\hline
$9_{47}$  & $ \{9,-6\}$ & $-1$ & $2$\\
\hline
$9_{48}$ & $\{11,-4\}$ & $-3/2$ & $3/2$\\
\hline
$9_{49}$ & $ \{15,0\}$ & $-3/2$ & $3/2$\\
\hline
 \end{tabular}
\vskip 0.3in 
\caption{ The knots up to nine crossings of period two.}

\end{center}
\end{table}
\begin{corollary} \label{low} Let $K$ be any knot of Table 1 and let $K'$ be an iterated cable of $K$.
Then $K'$ satisfies the Slope Conjecture.
\end{corollary}

\begin{proof}
Due to different conventions and normalizations of the colored Jones polynomial, $d_+[J_K(n)]$ (resp. $d_-[J_K(n)]$) in our paper is different from $\delta_K(n)$ (resp. $\delta_K^*(n)$) in \cite{Ga-slope}. For $n>0$ we have
\begin{eqnarray*}
d_+[J_K(n)] &=&  \delta_K(n-1) + (n-1)/2,\\
d_-[J_K(n)] &=&  \delta_K^*(n-1) + (1-n)/2.
\end{eqnarray*}
 
Consider the non-alternating knot $K=8_{19}$. By \cite[Section 4]{Ga-slope} we have
\begin{eqnarray*}
d_+[J_K(n)] &=& 3n^2-\big( 13+(-1)^n \big)/4,\\
d_-[J_K(n)] &=& 5(n-1)/2.
\end{eqnarray*} 
Moreover, $js_K = \{12\} \subset bs_K$ and $js^*_K = \{0\} \subset bs_K$.  In particular, the Jones slopes are integers and we have
$b(n) \le 0$ and $b^{*}(n) \ge 0$.  Similar analysis applies to the knots of Table 1.

Now given a cable $\K$ of $K$, since $\abs{q}>1$, $p/q$ is not a Jones slope of $K$.
 Theorem \ref{thm-quasi}  implies that $js_{\K} \subset bs_{\K}$  and that $js^*_{\K} \subset bs_{\K}$. In particular, $\K$ satisfies the Slope Conjecture. Applying this argument repeatedly we obtain the result for iterated cables.
\end{proof}

\section{Cabling knots with constant $a(n)$} 

\label{constant a(n)}

In Section 3 we studied the behavior of $d_+[J_K(n)]$ under knot cabling for knots of period at most two. In this section we study knots with period greater than two under the additional hypothesis
that, for $n \gg 0$ we have $a(n)=a$, where $a$ is a constant.
In this case, by abusing the terminology, we will say that $a(n)$ is constant. The main result in this section is the following.

 \begin{theorem}
 \label{thm-quasi-constant}
 Let $K$ be a knot such that for $n \gg 0$ we have
$$
d_+[J_K(n)]=an^2+b(n)n+d(n)
$$
where $a$ is a constant, $b(n)$ and $d(n)$ are periodic functions with $b(n) \le 0$. Let \begin{eqnarray*}
M_1 &=&  \max\{|b(i)-b(j)| \, : \,  i \equiv j \pmod{2}\},\\
M_2 &=& \max\{2b(i) + |b(i)-b(j)| + |d(i)-d(j)| \, : \, i \equiv j \pmod{2}\}.
\end{eqnarray*}

Suppose $p- (4a-M_1)q < 0$ or $p- (4a + M_1)q > \max\{0,M_2\}$. Then for $n \gg 0$ we have
$$d_+[J_{\K}(n)]=An^2+B(n)n+D(n)$$
where $A$ is a constant, and $B(n), D(n)$ are periodic functions with $B(n) \le 0$. Moreover, if $js_K \subset bs_K$ then $js_{\K} \subset bs_{\K}$.
 \end{theorem}
 
 As a corollary of Theorem \ref{thm-quasi-constant} and Theorem \ref{main:bs}  we obtain the following result which implies, in particular, that for knots with constant $a(n)$ the Slope Conjecture 
 is closed under cabling for infinitely many pairs $(p, q)$.
 \begin{corollary} \label{cables-a}
 Let $K$ be a knot such that for $n \gg 0$ we have
$$
d_+[J_K(n)]=an^2+b(n)n+d(n)
$$
where $a$ is a constant, $b(n)$ and $d(n)$ are periodic functions with $b(n) \le 0$. Let $M_1, M_2$
be as in the statement of Theorem \ref{thm-quasi-constant}.
If $K$ satisfies the Slope Conjecture,  then
$\K$ satisfies the conjecture  provided that  $p- (4a-M_1)q < 0$ or $p- (4a + M_1)q > \max\{0,M_2\}$. 
 \end{corollary}
\smallskip
 
 \begin{example}  

\label{3-knots} As an illustration we apply 
 Theorem \ref{thm-quasi-constant} and Corollary \ref{cables-a}  to the knots  $8_{20}, 9_{43}, 9_{44}$.  By  \cite[Section 4]{Ga-slope},
 these are the only  knots, with at most nine crossings, that have Jones period larger than 2. Indeed,
 the period of these knots is  $3$. Another application of Theorem \ref{thm-quasi-constant} will be illustrated in Example \ref{example-pretzel}.
 
 For $K=8_{20}$ we have
 \begin{eqnarray*}
d_+[J_K(n)] &=& \begin{cases} 2n^2/3 - n/2 - 1/6 &\mbox{if } n \not\equiv 0 \pmod{3}\\ 
2n^2/3 - 5n/6 -1/2 & \mbox{if } n \equiv 0 \pmod{3}.\end{cases} 
\end{eqnarray*}
Hence $K_{p,q}$ satisfies the Slope Conjecture if $p-\frac{7}{3}q<0$ or $p-3q > 0$.

For $K=9_{43}$ we have
 \begin{eqnarray*}
d_+[J_K(n)] &=& \begin{cases} 8n^2/3 - n/2 - 13/6 &\mbox{if } n \not\equiv 0 \pmod{3}\\ 
8n^2/3 - 5n/6 -7/2 & \mbox{if } n \equiv 0 \pmod{3}.\end{cases}
\end{eqnarray*}
Hence $K_{p,q}$ satisfies the Slope Conjecture if $p-\frac{31}{3}q<0$ or $p - 11q>\frac{2}{3}$.

For $K=9_{44}$ we have
\begin{eqnarray*}
d_+[J_K(n)] &=& \begin{cases} 7n^2/6 - n - 1/6 &\mbox{if } n \not\equiv 0 \pmod{3}\\ 
7n^2/6 - 4n/3 -1/2 & \mbox{if } n \equiv 0 \pmod{3}.\end{cases}
\end{eqnarray*}
Hence $K_{p,q}$ satisfies the Slope Conjecture if $p-\frac{13}{3}q<0$ or $p - 5q > 0 $.
\end{example}

Theorem \ref{thm-quasi-constant} follows from Theorem \ref{bdry slope} and the following proposition.

\begin{proposition}
\label{quasi-constant}
Let $K$ be a knot such that for $n \gg 0$ we have
$$
d_+[J_K(n)]=an^2+b(n)n+d(n)
$$
where $a$ is a constant, $b(n)$ and $d(n)$ are periodic functions with $b(n) \le 0$. Let
\begin{eqnarray*}
M_1 &=&  \max\{|b(i)-b(j)| : i  \equiv j \pmod{2} \},\\
M_2 &=& \max\{2b(i) + |b(i)-b(j)| + |d(i)-d(j)| : i \equiv j \pmod{2}\}.
\end{eqnarray*}

Suppose $p- (4a-M_1)q < 0$ or $p- (4a + M_1)q > \max\{0,M_2\}$. Then for $n \gg 0$ we have
$$d_+[J_{\K}(n)]=An^2+B(n)n+D(n)$$
where $A$ is a constant with $A \in \{q^2a, pq/4\}$, and $B(n), D(n)$ are periodic functions with $B(n) \le 0$.
\end{proposition}

\begin{proof} Fix $n \gg 0$. Recall the cabling formula \eqref{cables} of the colored Jones polynomial
$$J_{K_{p,q}}(n)= v^{pq(n^2-1)/4} \sum_{k \in \CS_n} v^{-pk(qk+1)} J_K(2qk+1).$$
In the above formula, there is a sum. As in the proof of Proposition \ref{quasi} we will show that, under the assumption of the proposition, there is a unique term of the sum whose highest degree is strictly greater than those of the other terms. This implies that the highest degree of the sum is exactly equal to the highest degree of that unique term. 

For $k \in \CS_n$ let $$f(k) := d_+ [v^{-pk(qk+1)} J_K(2qk+1)] = -pk(qk+1)+d_+[J_K(|2qk+1|)].$$
The goal is to show that $f(k)$ attains its maximum on $\CS_n$ at a unique $k$. 

Since $d_+[J_K(n)]$ is a quadratic quasi-polynomial, $f(k)$ is a piece-wise quadratic polynomial. The above goal will be achieved in 2 steps. In the first step we show that $f(k)$ attains its maximum on each piece at a unique $k$. Then in the second step we show that the maximums of $f(k)$ on all the pieces are distinct.
\smallskip

\textbf{Step 1.} Let $\pi$ be the period of $d_+[J_K(n)]$. For $\ve \in \{\pm 1\}$ and $0 \le i < \pi$, let $h_i^{\ve}(x)$ be the quadratic real polynomial defined by
$$
h_i^{\ve}(x) := ( -pq + 4q^2 a ) x^2 + ( -p + 4q a + 2q b(i) \ve ) x +  a + b(i) \ve + d(i).
$$
For each $k \in \CS_n$, we have $f(k)= h^{\ve_k}_{i_k}(k)$
for a unique pair $(\ve_k, i_{k})$. Let
$$I_n := \{ (\ve_k,i_{k}) \mid k \in \CS_n\}.$$
Then $f(k)$ is a piece-wise quadratic polynomial of exactly $|I_n|$ pieces, each of which is associated with a unique pair $(\ve, i)$ in $I_n$.

For each $(\ve,i) \in I_n$, let $$\CS_{n,\ve,i} := \{ k \in \CS_n \mid (\ve_k, i_k) = (\ve,i)\}$$
which is the set of all $k$ on the piece associated with $(\ve,i)$. 

The quadratic polynomial $h_i^{\ve}(x)$ is concave up if $p-4qa<0$, and concave down if $p-4qa>0$.  Hence, for $n \gg 0$, $h_i^{\ve}(k)$ is maximized on the set $\CS_{n,\ve,i}$ at a unique $k=k_{n,\ve,i}$, where
$$ k_{n,\ve,i} := \begin{cases} \max \CS_{n,\ve,i} &\mbox{if } (p - 4qa)\ve < 0, \\
\min \CS_{n,\ve,i} &\mbox{if } (p - 4qa)\ve > 0.  \end{cases}$$
Note that, as in the proof of Proposition \ref{quasi}, we use the assumption that $b(i) \le 0$ when $(p - 4qa)\ve > 0$. Moreover we have 
$$ \begin{cases}  |k_{n,\ve,i}|  \to \infty \mbox{ as }  n \to \infty, &\mbox{if } p - 4qa < 0 \\
 |k_{n,\ve,i}|  \le \pi, &\mbox{if } p - 4qa > 0.  \end{cases}$$
\smallskip

\textbf{Step 2.} Let 
$$\text{Max}_n := \max \{f(k) \mid k \in \CS_n\}.$$
From step 1 we have $\text{Max}_n =  \max \{ h_{i}^{\ve}(k_{n, \ve, i}) \mid (\ve, i) \in I_n \}.$ We claim that $$h_{i_1}^{\ve_1}(k_{n,\ve_1,i_1}) \not= h_{i_2}^{\ve_2}(k_{n,\ve_2,i_2})$$
for $(\ve_1,i_1) \not= (\ve_2,i_2)$. 

Indeed, let $k_1:=k_{n,\ve_1,i_1}$ and $k_2:=k_{n,\ve_2,i_2}$. Note that $k_1 \not= k_2$. Moreover, $k_1$ and $k_2$ are both in  $\BZ$ or $\frac{1}{2} + \BZ$. As a result we have $k_1 \pm  k_2 \in \BZ$, and $i_1 - i_2 \equiv 2q (k_1 - k_2) \equiv 0 \pmod{2}$. Let
$$\sigma := h_{i_1}^{\ve_1}(k_1) - h_{i_2}^{\ve_2}(k_2).$$

Without loss of generality, we can assume that $|k_1| \ge |k_2|$. Then we write $\sigma = \sigma' + d(i_1) - d(i_2)$ where
$$
\sigma' := \begin{cases} (k_1 - k_2) \Big( (-p+4qa) \big( q(k_1+k_2)+1 \big) + 2qb(i_1)\ve_1 \Big) \\ \qquad \qquad \qquad + \, \big( b(i_1) - b(i_2) \big) \ve_1 (2qk_2 + 1) &\mbox{if } \ve_1 = \ve_2, \\
\Big( (-p+4qa)(k_1 - k_2) + 2b(i_1)\ve_1 \Big) \big( q(k_1+k_2)+1 \big) 
\\ \qquad \qquad \qquad - \, \big( b(i_1) - b(i_2) \big) \ve_1 (2qk_2 + 1) &\mbox{if } \ve_1 \not= \ve_2.  \end{cases}
$$
We consider the following 2 cases.
\medskip

\textit{Case 1:} Suppose that $p - (4a - M_1)q <0 $. In particular, we have $p - 4qa <0$. There are 2 subcases.
\smallskip

\underline{Subcase 1.1:} We have $\ve_1 = \ve_2$. Since $k_1$ and $k_2$ have the same sign, we have 
$$|q(k_1+k_2)+1| - |2qk_2+1| = 2q (|k_1| - |k_2|) \ge 0.$$
Hence
\begin{eqnarray*}
|\sigma'| &\ge& \big| (-p+4qa) \big( q(k_1+k_2)+1 \big) + 2qb(i_1)\ve_1 \big| - \big| \big( b(i_1) - b(i_2) \big) \big( q(k_1+k_2)+1 \big) \big| \\
&\ge& \big( -p + 4qa - |b(i_1) - b(i_2)| \big) |q(k_1+k_2)+1| + 2q b(i_1).
\end{eqnarray*}
Since $|q(k_1+k_2)+1| \to \infty$ as $n \to \infty$, and $$-p + 4qa - |b(i_1) - b(i_2)| \ge -p + 4qa - M_1 >0$$ 
we get $|\sigma'| \to \infty$ as $n \to \infty.$

\smallskip

\underline{Subcase 1.2: } We have $\ve_1 \not= \ve_2$. Since $k_1$ and $k_2$ have opposite signs, we have 
$$(q|k_1-k_2| + 1) - |2qk_2+1| \ge 2q (|k_1| - |k_2|)  \ge 0.$$
Hence
\begin{eqnarray*}
|\sigma'| &\ge& \big|  (-p+4qa)(k_1 - k_2) + 2b(i_1)\ve_1 \big| -  | b(i_1) - b(i_2)| \, ( q|k_1 - k_2| + 1) \\
&\ge& \big( -p + 4qa - q |b(i_1) - b(i_2)| \big) |k_1 - k_2| -  | b(i_1) - b(i_2)|  + 2 b(i_1).
\end{eqnarray*}
Since $|k_1 - k_2| \to \infty$ as $n \to \infty$, and $$-p + 4qa - q |b(i_1) - b(i_2)| \ge -p + 4q a  - q M_1 >0,$$ we get $|\sigma'| \to \infty$ as $n \to \infty.$

\medskip 

\textit{Case 2:} Suppose that $p- (4a + M_1)q > \max\{0, M_2\}$. There are 2 subcases.
\smallskip

\underline{Subcase 2.1:} We have $\ve_1 = \ve_2$. Note that $q(k_1+k_2)+1$ and $\ve_1$ have the same sign. Moreover, both $-p+4qa$ and $2qb(i_1)$ are non-positive. As in subcase 1.1 we have 
\begin{eqnarray*}
|\sigma'| &\ge& \big| (-p+4qa) \big( q(k_1+k_2)+1 \big) + 2qb(i_1)\ve_1 \big| - \big| \big( b(i_1) - b(i_2) \big) \big( q(k_1+k_2)+1 \big) \big| \\
&=& \big( p - 4qa - |b(i_1) - b(i_2)| \big) \, |q(k_1+k_2)+1| - 2q b(i_1).
\end{eqnarray*}
Since $p - 4qa - |b(i_1) - b(i_2)|  \ge p - 4qa - M_1 > \max \{0, M_2\}$, we get 
$$|\sigma'| > M_2  -2 q b(i_1) \ge |d(i_1) - d(i_2)|.$$
\smallskip

\underline{Subcase 2.2:} We have $\ve_1 \not= \ve_2$. Note that $k_1 - k_2$ and $\ve_1$ have the same sign. Moreover, both $-p+4qa$ and $2qb(i_1)$ are non-positive. As in subcase 1.2 we have
\begin{eqnarray*}
|\sigma'| &\ge& \big|  (-p+4qa)(k_1 - k_2) + 2b(i_1)\ve_1 \big| -  | b(i_1) - b(i_2)| \, ( q|k_1 - k_2| + 1) \\
&=& \big( p - 4qa - q |b(i_1) - b(i_2)| \big) |k_1 - k_2| - | b(i_1) - b(i_2)|  - 2 b(i_1).
\end{eqnarray*}
Since $p - 4qa - q |b(i_1) - b(i_2)| \ge p - 4qa - q M_1 > \max\{0, M_2\}$, we get $$|\sigma'| > M_2 -  | b(i_1) - b(i_2)|  -2 b(i_1) \ge |d(i_1) - d(i_2)|.$$

In all cases, for $n \gg 0$ we have $|\sigma'| > |d(i_1)-d(i_2)|.$ Hence $$\sigma = \sigma' + d(i_1) - d(i_2) \not= 0.$$
We have proved that $f(k)$ attains its maximum on $\CS_n$ at a unique $k$. More precisely, there exists a unique $(\ve_n, i_n) \in I_n$ such that $h_{i_n}^{\ve}(k_{n, \ve_n, i_n}) = \text{Max}_n$. 

Equation \eqref{cables} then implies that  
\begin{eqnarray*}
d_+[J_{K_{p,q}}(n)] &=& pq(n^2-1)/4 + h^{\ve_n}_{i}(k_{n,\ve_n,i_n}).
\end{eqnarray*}

If $p - 4qa < 0$ then $k_{n,\ve_n,i_n} = \ve_n (n/2 + s_n)$, where $s_n$ is a periodic sequence and $s_n \le -1/2$. We have
\begin{eqnarray*}
d_+[J_{K_{p,q}}(n)] &=& q^2a n^2 + \big( (-p + 4qa) (q s_n + \ve_n /2) + qb(i_n) \big) n -  pq/4 \\
&& + \, (-p + 4qa)s_n (q s_n + \ve_n) + 2q b(i_n) s_n+  a + b(i_n) \ve_n + d(i_n).
\end{eqnarray*}
In this case we have $$B(n)=(-p + 4qa) (q s_n + \ve_n /2) + qb(i_n) < 0,$$ since $q s_n + \ve_n /2 \le -q/2 + 1/2 <0$ and $b(i_n) \le 0$.

If $p - 4qa > 0$ then $k_{n,\ve_n,i_n} = s_n$, where $s_n$ is a periodic function. We have
\begin{eqnarray*}
d_+[J_{K_{p,q}}(n)] &=& pq (n^2-1)/4 
+ (-p + 4qa)s_n (q s_n + 1) + 2q b(i_n) \ve_n s_n\\
&& + \,  a + b(i_n) \ve_n + d(i_n).
\end{eqnarray*}
In this case we have $B(n)=0$.

This completes the proof of Proposition \ref{quasi-constant}.
\end{proof}

\section{Conjectures} Recall that for every knot $K \subset S^3$, there is an integer $N_K>0$ and periodic functions $a_{K}(n), b_{K}(n), c_{K}(n)$ such that 
$$d_+[J_K(n)] = a_{K}(n) \, n^2 + b_{K}(n) n + c_{K}(n)$$
for $n \ge N_K$.   In Propositions \ref{quasi}  and \ref{quasi-constant} we made the assumption that  $b(n) \le 0$. Then we concluded that, under the appropriate hypotheses, this property
is preserved under cabling. As we will discuss below, the property that $b(n)\leq 0$ is known to hold  
for all non-trivial knots, of any period,  for which $b_K(n)$ has been calculated.

We propose the following conjecture. 

\begin{conjecture}
\label{conj}
For every non-trivial knot $K \subset S^3$, we have $$b_K(n) \le 0.$$
Moreover, if $b_K(n)=0$ then   $K$ is a composite knot or  a cable knot or a torus knot. 
\end{conjecture}
Note that $b_U(n)=1/2$ for the trivial knot $U$. 

\begin{remark} It is known that a knot $K$ is composite or cable or a torus knot if and only if its complement $M_K$ contains  embedded essential annuli \cite[Lemma V.1.3.4]{jacoshalen}. Thus the last part of Conjecture \ref{conj} can alternatively be stated as
follows: If $b_K(n)=0$, then $M_K$ contains an embedded essential  annulus.
\end{remark}

By Theorem \ref{slopeadequate}, for $B$-adequate knots and their  iterated cables we have $b_K(n) \le 0$.
Moreover, if $K$ is a $B$-adequate knot and $2b=v_B-c=0$, then  by Lemma \ref{v_B-c}
$K$ is a torus knot.
Thus Conjecture \ref{conj} holds for $B$-adequate knots and their cables.
 Notice that,
as shown in the proof of Theorem \ref{slopeadequate}, the case $b_K(n)=0$ occurs quite often for cables of $B$-adequate knots.
Conjecture \ref{conj} holds for the knots of  Table 1. In the next section we will check that Conjecture \ref{conj} holds true for 
 2-fusion knots.
 Thus the conjecture holds for all
 the classes of knots for which
$ d_+[J_{K}(n)] $ and $ d_-[J_{K}(n)]$ have been calculated to date.
\medskip

We now turn our attention to the Strong Slope Conjecture (Conjecture \ref{slopeaug})  stated in the Introduction.
By Theorem \ref{slopeadequate}, Conjecture \ref{slopeaug} is true for iterated cables of $B$-adequate knots (and in particular iterated torus knots).
Furthermore,  the arguments in the proofs of Corollary \ref{integers} and  Theorem \ref{slopeadequate} generalize easily to show that, under the hypothesis of Proposition \ref{quasi} or Proposition \ref{quasi-constant}, the Strong Slope Conjecture is closed under knot cabling. For instance we have the following:

\begin{corollary}
Let $K$ be a knot such that for $n \gg 0$ we have
$d_+[J_K(n)]$
is a quadratic quasi-polynomial of period $\le 2$. Suppose $p/q$ is not a Jones slope of $K$. Then if $K$ satisfies Conjecture \ref{slopeaug} so does $K_{p/q}$.
\end{corollary}

\begin{remark}Similar ones to Conjectures \ref{conj} and \ref{slopeaug}  can also be formulated  for the lowest degree of the colored Jones polynomial by noting that $J_K(n, v) = J_{K^*}(n, v^{-1})$, where $K^*$ is the mirror image of $K$. To illustrate this
point  we discuss the example of the knot $9_{49}$. This knot has genus two and is not $B$-adequate since the leading coefficient of its colored Jones polynomial 
is 2 \cite{Ga-slope}
and not $\pm 1$.  By Table 1, $ js^*_K=\{0\} $ and $2b_K^*(n)=3$.
A genus two Seifert surface $S$,  has boundary  slope $0=a_K^*(n)$ and $\chi(S)=-3=-2b_K^*(n)$. Thus   
Conjectures  \ref{conj} and \ref{slopeaug} 
hold for
the mirror image $9^{*}_{49}$. We also mention that  the same is true for $9_{49}$ since it is known to be $A$--adequate.
\end{remark}

 Next we 
discuss more families of knots, not covered by Theorem  \ref{slopeadequate},
for which the above conjectures are true.

\subsection{Non-alternating Montesinos knots up to nine crossings} Table 2 summarizes the relevant information 
about these knots.
The Jones slopes and  the sets of cluster points 
$\{2b_K(n)\}', \{2b_K^{*}(n)\}'$ were obtained from Garoufalidis' paper \cite{Ga-slope}.
The corresponding
boundary slopes together with the values $\chi(S)$ and $\abs{\partial S}$ were obtained using Dunfield's program for calculating boundary slopes of Montesinos knots \cite{ndunfield}.
In all cases,  Conjectures \ref{conj} and \ref{slopeaug} are easily verified for the knots and their mirror images.
Note that, for example, for $9_{44}$
we have $a/b=14/3 \in js_K$, $\{2b_K(n)\}'=\{-2, -8/3\}$, and $\frac{\chi(S)}{{\abs{\partial S} b}} =\frac{-6}{3}=-2$ as predicted
by  Conjecture \ref{slopeaug}. However, this assertion alone  doesn't guarantee that $b_K(n)\le 0$. Thus,
in general,   Conjecture \ref{slopeaug} does not imply Conjecture \ref{conj}.

\begin{table}[]
\begin{center}
\begin{tabular}{|c|c|c|c|c|c|c|c|c|}
\hline
$K$& $js_K$ & $jx_{K}$ & $\chi(S)$ & $|\partial{S}|$ & $js^*_K$ & $jx_K^*(n)$ & $\chi(S^*)$ & $|\partial{S^*}|$\\
\hline
$8_{19}$ & $\{12\}$ &   $\{0\}$ & $0$ & $2$ & $\{0\}$ & $\{5\}$ & $-5$ & $1$ \\
\hline
$8_{20}$ &$\{ 8/3\}$  &$\{-1, -5/3\}$  & $-3$ & $1$ & $\{-10\}$& $\{4\}$ &$-4$ &$1$\\
\hline
$8_{21}$ & $\{1\}$ &  $\{-2\}$ & $-4$ & $2$ & $\{-12\}$ & $\{3\}$ &$-3$ &$1$\\
\hline
$9_{42}$  & $\{6\}$ &  $\{-1\}$ & $-2$ & $2$ & $\{-8\}$ & $\{5\}$ &$-5$ &$1$\\
\hline
$9_{43}$  & $\{ 32/3\}$ &  $\{-1, -5/3\}$ & $-3$& $1$ & $\{-4\}$& $\{5\}$ &$-5$ &$1$\\
\hline
$9_{44}$ & $\{ 14/3\}$ & $\{-2, -8/3\}$  &$-6$ &$1$ & $\{-10\}$ & $\{4\}$ &$-4$ &$1$\\
\hline
$9_{45}$  & $\{ 1\}$ &  $\{-2\}$ &$-4$ &$2$ & $\{-14\}$& $\{4\}$ &$-4$&$1$\\
\hline
$9_{46}$  & $\{ 2\}$ &  $\{-1\}$ &$-2$ &$2$& $\{-12\}$ & $\{5\}$ &$-5$&$1$\\
\hline
$9_{48}$ & $\{11\}$ &  $\{-3\}$ &$-6$ & $2$ &$\{-4\}$& $\{3\}$ &$-3$&$1$\\
\hline
 \end{tabular}
\vskip 0.3in 
\caption{Non-alternating Montesinos knots up to nine crossings.}
\end{center}
\end{table}

\begin{corollary}\label{conjcables} Suppose that $K\in \{ 8_{19},  8_{21}, 9_{42}, 9_{45}, 9_{46},
9_{48}, 9_{49} \}$ and let $K'$ be an iterated cable of $K$. Then, Conjectures \ref{conj} and  \ref{slopeaug}
hold true for $K'$.
\end{corollary}
\begin{proof}
We first note that $
d_+[J_{K}(n)]=an^2+bn+d(n)
$
is a quadratic quasi-polynomial of period $\le 2$, with $4a \in \BZ, 2b \in \BZ$ and $b \le 0$. Suppose
${K_{p,q}}$ is cable of $K$.
Theorem \ref{thm-quasi} and the proof of Proposition \ref{quasi} then imply that 
$$
d_+[J_{K_{p,q}}(n)]=An^2+Bn+D(n)
$$
is a quadratic quasi-polynomial of period $\le 2$, with $4A \in \BZ, 2B \in \BZ$ and $B \le 0$.
Moreover, the proof of Proposition \ref{quasi} shows that one of the following is true:
\begin{enumerate}
\item We have $4A=pq$ and $B=0$.
\item We have $4A=4q^2a$ and $2B=2 \abs{q} b+ (1- \abs{q}) \abs{4aq-p}$.

\end{enumerate}
 In case (1), a surface $S$ with boundary slope $pq$ is the cabling annulus; thus $\chi(S)=0=2B$
and Conjectures \ref{conj} and  \ref{slopeaug} are satisfied.
In case (2),  let $S$ be a surface that satisfies Conjecture \ref{slopeaug}
for $K$. We view $\M$ as the union of $M_K$ and a cable space $\C$.
An essential  surface $S'$ realizing the boundary slope $4A=4q^2a$ for $\K$ 
is obtained by Theorem \ref{main:bs}. This surface is constructed as in the proof of Lemma \ref{glue}.
By  Corollary \ref{integers} we see that $\abs{\partial S'}=\abs{\partial S}$,
$\abs {S'\cap T_{+}}=\abs{q} \, \abs{\partial S}$ and
$$\chi(S')=\chi(M_K\cap S')+\chi(\C \cap S')=\abs{q}  \chi(S)+\abs{\partial S} \, \abs{4aq-p} (1- \abs{q} ).$$
By assumption, $\chi(S)=\abs{\partial S} (2b)$.
Combining the last  two equations with the formula in (2) above, we have
$$\chi(S')=\abs{\partial S} \left( 2 \abs{q} b+ (1-\abs{q}) \abs{4aq-p}\right)=\abs{\partial S'}(2B),$$
which shows that $S'$ satisfies Conjecture \ref{slopeaug}.

Applying the above argument repeatedly we obtain the result for iterated cables.
\end{proof}

\subsection{A family of pretzel knots}
 Consider the pretzel knot $K_p=(-2,3,p)$, where $p$ is an odd integer. It is known that $K_p$ is $A$-adequate if $p>0$, and $B$-adequate if $p<0$. Moreover, $K_p$ is a torus knot if $p \in \{1,3,5\}$.

Suppose now that  $p \ge 5$. Then $K_p$ is $A$-adequate and  by above discussion Conjecture  \ref{slopeaug} 
holds  for
the mirror image $K_p^{*}$.

By \cite[Section 4.7]{Ga-slope} and Example 6.3 below we have $$4a_{K_p}(n) = 2(p^2-p-5)/(p-3) \quad \text{and} \quad 2b_{K_p}(n) = 
-(p-5)/(p-3).$$

Since $4a_{K_p}(n)$ is not an integer, in fact it is easily checked that $ \gcd ( 2(p^2-p-5), p-3)=1$, $K_p$ is not $B$-adequate.
Thus we can't apply Theorem  \ref{slopeadequate}. Note moreover that $\gcd(p-5, p-3)=2$.
By  \cite[Theorem 3.3]{callahan}, $K_p$ has an essential surface with boundary slope $2(p^2-p-5)/(p-3)$, with two boundary components, and Euler characteristic $-(p-5)$, which is equal to $(p-3)(2b_{K_p}(n))$.   
Note that for $p=5$ we get $b_{K_p}(n)=0$. The knot $K=(2,-3,5)$ is known to be a torus knot, which is in agreement with Conjecture \ref{conj}.

Suppose now that $p \le -1$. Then $K$ is $B$-adequate. By \cite[Section 4.7]{Ga-slope} and Example 6.3
we have $4a^*_{K_p}(n) = 2(p+1)^2/p$ and $$2b^*_{K_p}(n) = 
\begin{cases}
1 &\mbox{if } n \not\equiv 0 \pmod{p} \\ 
1 - 2/p & \mbox{if } n \equiv 0 \pmod{p}.
\end{cases}$$

Again since $4a^*_{K_p}(n)$ is not an integer, $K_p$ is not $A$-adequate and Theorem  \ref{slopeadequate} doesn't apply to $K_p^*$.
According to \cite[Theorem 3.3]{callahan}, 
 $K_p$ has an essential surface with boundary slope $2(p+1)^2/p$ and Euler characteristic $p$, which is equal to $p(2b_{K_p}(n))$ when $n \not\equiv 0 \pmod{p}$. Thus we have:
 
\begin{corollary} For an odd integer $p$,  the pretzel knots  $K_p=(-2,3,p)$ and $K^{*}_p$
satisfy Conjectures \ref{conj} and \ref{slopeaug}. 
\end{corollary}

\section{Two-fusion knots}

The family of 2-fusion knots is a two-parameter family  of closed 3-braids denoted
by
$$\{ K(m_1, m_2)  \mid  m_1, m_2\in \ZZ\}.$$
For the precise definition and description see \cite {GV, DG}. The purpose of this section is to prove the following.

\begin{theorem}
\label{thm:slope-cable} Conjecture  \ref{conj}  holds  for 2-fusion knots. 
\end{theorem}

Note that $K(m_1, m_2)$ is a torus knot if $m_2 \in \{-1,0\}$. In fact, $K(m_1, 0) = T(2, 2m_1 + 1)$ and $K(m_1, -1) = T(2, 2m_1 -3)$. It is known that $K(m_1, m_2)$ is hyperbolic if $m_1 \notin \{0, 1\}$, $m_2 \notin \{-1,0\}$ and $(m_1, m_2) \not= (-1,1)$. See \cite{GV}. Note that $K(-1,1) = T(2,5)$.

From now on we consider $m_2 \notin \{-1,0\}$ only.  For $n \in \BN$ and $k_1, k_2 \in \BZ$ such that $0 \le k_1 \le n$ and $|n-2k_1| \le n+2k_2 \le n+2k_1$, let
\begin{eqnarray*}
Q(n,k_1,k_2)&=& \frac{k_1}{2} - \frac{3 k_1^2}{2} - 3 k_1 k_2 - k_2^2 - k_1 m_1 - k_1^2 m_1 - k_2 m_2 - k_2^2 m_2 - 6 k_1 n \\ 
& & - \, 3 k_2 n + 2 m_1 n + 4 m_2 n - k_2 m_2 n - 2 n^2 + m_1 n^2 + 2 m_2 n^2 \\ 
& & + \, \frac{1}{2} \left( (1 + 8 k_1 + 4 k_2 + 8 n) \min\{l_1,l_2,l_3\} - 3 \min\{l_1,l_2,l_3\}^2\right)
\end{eqnarray*}
where
$$
l_1=2 k_1 + n, \qquad l_2=2 k_1 + k_2 + n, \qquad l_3=k_2 + 2 n.
$$
\medskip

\subsection{ The highest degree.} The quantity  $Q(n,k_1,k_2)$ is closely related
to the degree $\delta_K(n)$. According to \cite{GV}  for the 2-fusion knot $K=K(m_1,m_2)$, with $m_2 \notin \{-1, 0\}$, we have the following possibilities:
\vskip 0.04in

\textbf{Case A.} Suppose that $m_1, m_2 \ge 1$. Then
$$\delta_K(n) = Q(n,k_1,-k_1),$$

\noindent where $$c_1=\frac{1 - m_1 + m_2 + m_2 n}{2 (-1 + m_1 + m_2)},$$ 
and  $k_1$ is of the integers closest to $c_1$ satisfying $k_1 \le n/2$. 

\vskip 0.04in

\textbf{Case B.} Suppose that $m_1 \le 0, m_2 \ge 1$. There are 2 subcases.

\textbf{(B-1)} We have ($1 + m_1 + m_2 \le 0$) or ($1 + m_1 + m_2 > 0$ and $1 + 2m_1 + m_2 < 0$). Then
$$\delta_K(n) = Q(n,n,0).$$

\textbf{(B-2)}  We have $1 + m_1 + m_2 > 0$ and $1 + 2m_1 + m_2 \ge 0$. 
 Then
$$\delta_K(n) = Q(n,k_1,k_1-n),$$
\noindent where  $$c_2=\frac{1 - m_1 - m_2 + (1 + m_2) n}{2 (1 + m_1 + m_2)}.$$
and $k_1$ is one of the integers closest to $c_2$.

\textbf{Case C.} Suppose that  $m_2 \le -2$. There are 2 subcases.

\textbf{(C-1)} We have $m_1 \le -3m_2/2$. Then 
$$\delta_K(n) = Q(n,n,n).$$

\textbf{(C-2)} We have $m_1 > -3m_2/2$. Let $$c_3=\frac{-3/2 + m_1 + m_2 + (1+m_2 ) n}{1-2m_1-2m_2}$$
and let $k_1$ be one of the integers closest to $c_3$. Then
$$\delta_K(n) = \begin{cases} Q(n,k_1,k_1) &\mbox{if } c_3 \notin \frac{1}{2} + \BZ, \\ 
Q(n,k_1,k_1)-(c_3 + 1/2) & \mbox{if } c_3 \in \frac{1}{2} + \BZ. \end{cases}$$

\subsection{Calculating the linear term}  In this subsection we will prove the following.

\begin{theorem}

\label{prop:b}

For the 2-fusion knot $K=K(m_1,m_2)$, with $m_2 \notin \{-1,0\}$, we have
$$
b_K(n) = \begin{cases} \frac{m_2(1-m_1)}{2(-1+m_1+m_2)} &\mbox{if } (m_1, m_2) \in I_1 \cup I_2, 
\smallskip
\\ 1 + m_1  & \mbox{if } (m_1, m_2) \in I_3, 
\smallskip
\\ \frac{m_1(m_2-1)}{2(1+m_1+m_2)}   & \mbox{if } (m_1, m_2) \in I_4, 
\medskip
\\ 5/2 + m_1 + 3 m_2  & \mbox{if } (m_1, m_2) \in I_5, 
\smallskip
\\ \frac {(-5 + 2m_1)(1 + m_2)} {2 (-1 + 2 m_1 + 2 m_2)} & \mbox{if } (m_1, m_2) \in I_6 \mbox{ and } \frac {-1+(1 + m_2)(n-1)} {-1 + 2 m_1 + 2 m_2} \notin  \BZ,
\medskip
\\ \frac {(-3 + 2m_1)(1 + m_2)} {2 (-1 + 2 m_1 + 2 m_2)} & \mbox{if } (m_1, m_2) \in I_6\mbox{ and } \frac {-1+(1 + m_2)(n-1)} {-1 + 2 m_1 + 2 m_2} \in  \BZ.
\end{cases}
$$
In particular we have $b_K(n) \le 0$. Moreover $b_K(n) = 0$ if and only if $m_1 \in \{0,1\}$ and $m_2 \ge 1$, or $(m_1, m_2) = (-1,1)$.
\end{theorem}

\begin{proof} As in the previous subsection, there are 3 cases.

\medskip

\textbf{Case A.} $m_1, m_2 \ge 1$. Recall that $$c_1=\frac{1 - m_1 + m_2 + m_2 n}{2 (-1 + m_1 + m_2)}$$ 
and $k_1$ is one of the integers closest to $c_1$ satisfying $k_1 \le \frac{n}{2}$. We have
\begin{eqnarray*}
\delta_K(n) = Q(n,k_1,-k_1)
&=& (1-m_1-m_2) k_1^2 + ( 1 - m_1 + m_2 + m_2 n) k_1 \\
&& + \, 2 m_1 n + 4 m_2 n  + m_1 n^2 + 2 m_2 n^2 + \frac{n}{2} + \frac{n^2}{2}.
\end{eqnarray*}

Write $k_1 = c_1 + r_n$ where $r_n$ is a periodic sequence with $\begin{cases} |r_n| \le 1/2 &\mbox{if } m_1 \ge 2, 
\\ r_n \in \{-1/2, -1\}  & \mbox{if } m_1 = 1. 
\end{cases}$ We have
$$
\delta_K(n) = Q(n,c_1+r_n,-c_1-r_n) = Q(n,c_1,-c_1) + (1-m_1-m_2) r_n^2
$$
and
\begin{eqnarray*}
Q(n,c_1,-c_1) &=& 2 m_1 n + 4 m_2 n  + m_1 n^2 + 2 m_2 n^2 + \frac{n}{2} + \frac{n^2}{2} -\frac{( 1 - m_1 + m_2 + m_2 n)^2}{4(1-m_1-m_2)} \\
&=& \left( m_1 + 2m_2 + \frac{1}{2} + \frac{m_2^2}{4(-1 + m_1 + m_2)} \right) n^2 \\
&& + \,  \left( 2m_1 + 4m_2 + \frac{1}{2} + \frac{m_2(1-m_1+m_2)}{2(-1 + m_1 + m_2)} \right) n - \frac{( 1 - m_1 + m_2)^2}{4(1-m_1-m_2)}.
\end{eqnarray*} 

Since $d_+[J_K(n)] = \delta_K(n-1) + (n-1)/2$ we obtain
\begin{eqnarray*}
d_+[J_K(n)] &=& \left( m_1 + 2m_2 + \frac{1}{2} + \frac{m_2^2}{4(-1 + m_1 + m_2)} \right) n^2 + \frac{m_2(1-m_1)}{2(-1+m_1+m_2)} n\\
&& - \, \left( m_1 + 2m_2 + \frac{1}{2} - \frac{(1-m_1)^2}{4(-1 + m_1 + m_2)} \right) + (1-m_1-m_2) r^2_{n-1}.
\end{eqnarray*} 
\medskip

\textbf{Case B.} $m_1 \le 0, m_2 \ge 1$.  There are 2 subcases.

\textbf{(B-1)} ($1 + m_1 + m_2 \le 0$) or ($1 + m_1 + m_2 > 0$ and $1 + 2m_1 + m_2 < 0$). Then
$$\delta_K(n) = Q(n,n,0)=\left( \frac{1}{2} + 2 m_2 \right) n^2 + \left( \frac{3}{2} + m_1 + 4 m_2 \right) n.$$
Hence $$
d_+[J_K(n)] = \left( \frac{1}{2} + 2 m_2 \right) n^2 + (1 + m_1) n - \left( \frac{3}{2} + m_1 + 2 m_2 \right).
$$

\textbf{(B-2)} $1 + m_1 + m_2 > 0$ and $1 + 2m_1 + m_2 \ge 0$. Recall that $$c_2=\frac{1 - m_1 - m_2 + (1 + m_2) n}{2 (1 + m_1 + m_2)}$$
and $k_1$ is one of the integers closest to $c_2$. We have 
\begin{eqnarray*}
\delta_K(n) = Q(n,k_1,k_1-n)
&=& (-1-m_1-m_2) k_1^2 + ( 1 - m_1 - m_2 + (1 + m_2) n) k_1 \\
&& + \, 2 m_1 n + 5 m_2 n  + m_1 n^2 + 2 m_2 n^2 + \frac{n}{2} + \frac{n^2}{2}.
\end{eqnarray*}
Write $k_1 = c_2 + r_n$ where $r_n$ is a periodic sequence with $|r_n| \le 1/2$.  As in Case A we have 
$$\delta_K(n) = Q(n,c_2,c_2-n) + (-1-m_1-m_2)r_n^2$$
and
\begin{eqnarray*}
Q(n,c_2,c_2-n) &=&  2 m_1 n + 5 m_2 n  + m_1 n^2 + 2 m_2 n^2 + \frac{n}{2} + \frac{n^2}{2} - \frac{( 1 - m_1 - m_2 + (1 + m_2) n)^2}{4(-1-m_1-m_2)} \\
&=& \left( \frac{3}{4} + \frac{3m_1}{4} + \frac{9m_2}{4} +  \frac{m_1^2}{4(1 + m_1 + m_2)} \right) n^2  \\
&& + \,  \left( 1 + 2m_1 + \frac{9m_2}{2}  - \frac{m_1}{1 + m_1 + m_2} \right) n - \frac{( 1 - m_1 - m_2)^2}{4(-1-m_1-m_2)}.
\end{eqnarray*} 
Hence 
\begin{eqnarray*}
d_+[J_K(n)] &=& \left( \frac{3}{4} + \frac{3m_1}{4} + \frac{9m_2}{4} +  \frac{m_1^2}{4(1 + m_1 + m_2)} \right) n^2 + \frac{m_1(m_2-1)}{2(1+m_1+m_2)} n\\
&& - \,  \left( \frac{3}{4} + \frac{3m_1}{4} + \frac{9m_2}{4} - \frac{(m_2 - 1)^2}{4(1 + m_1 + m_2)}\right) + (-1-m_1-m_2)r_n^2.
\end{eqnarray*}
\medskip

\textbf{Case C.} $m_2 \le -2$. There are 2 subcases.

\textbf{(C-1)} $m_1 \le -3m_2/2$. Then 
$$\delta_K(n) = Q(n,n,n)=(2 + m_1 + 3 m_2) n.$$
Hence $$d_+[J_K(n)] = (5/2 + m_1 + 3 m_2) (n-1).$$ 

\textbf{(C-2)} $m_1 > -3m_2/2$. Recall that $$c_3=\frac{-3/2 + m_1 + m_2 + (1+m_2 ) n}{1-2m_1-2m_2}$$
and let $k_1$ be one of the integers closest to $c_3$. We have
$$\delta_K(n) = \begin{cases} Q(n,k_1,k_1) &\mbox{if } c_3 \notin \frac{1}{2} + \BZ \\ 
Q(n,k_1,k_1)-(c_3 + 1/2) & \mbox{if } c_3 \in \frac{1}{2} + \BZ \end{cases}$$
and
\begin{eqnarray*}
Q(n,k_1,k_1)
&=& \left( 1/2-m_1-m_2 \right) k_1^2 - \left( -3/2 + m_1 + m_2 + (1+m_2 ) n \right) k_1 \\
&& + \, 2 m_1 n + 4 m_2 n  + m_1 n^2 + 2 m_2 n^2 + \frac{n}{2} + \frac{n^2}{2}.
\end{eqnarray*}

Write $k_1 = c_3 + r_n$ where $r_n$ is a periodic sequence with $|r_n| \le 1/2$. As in Case A we have 
$$Q(n,k_1,k_1) = Q(n,c_3,c_3) + \left( 1/2-m_1-m_2 \right) r_n^2$$
and
\begin{eqnarray*}
Q(n,c_3,c_3) &=&  2 m_1 n + 4 m_2 n  + m_1 n^2 + 2 m_2 n^2 + \frac{n}{2} + \frac{n^2}{2} - \frac{\left( -3/2 + m_1 + m_2 + (1+m_2 ) n \right)^2}{4(1/2-m_1-m_2)} \\
&=& \frac {(2 m_1 + 3 m_2)^2} {2 (-1 + 2 m_1 + 2 m_2)} n^2 + \left( \frac{1}{2} + 2 m_1 + \frac{9m_2}{2} + \frac {-3 + 2 m_1} {2 (-1 + 2 m_1 + 2 m_2)} \right) n\\
&& - \, \frac{\left( -3/2 + m_1 + m_2  \right)^2}{4(1/2-m_1-m_2)} .
\end{eqnarray*} 
Hence 
\begin{eqnarray*}
d_+[J_K(n)] &=& \frac {(2 m_1 + 3 m_2)^2} {2 (-1 + 2 m_1 + 2 m_2)} n^2 + 
\begin{cases} \frac {(-5 + 2m_1)(1 + m_2)} {2 (-1 + 2 m_1 + 2 m_2)} n &\mbox{if } \frac {-1+(1 + m_2)(n-1)} {-1 + 2 m_1 + 2 m_2}  \notin \BZ \medskip \\ 
\frac {(-3 + 2 m_1)(1 + m_2)} {2(-1 + 2 m_1 + 2 m_2)}n & \mbox{if } \frac {-1+(1 + m_2)(n-1)} {-1 + 2 m_1 + 2 m_2}  \in  \BZ \end{cases} \\
&& - \, \left( \frac{1}{2} + m_1 + 2m_2 - \frac{\left( 2m_1 -5  \right)^2}{8(2m_1 + 2m_2 -1)} \right) + \left( 1/2-m_1-m_2 \right) r_{n-1}^2.
\end{eqnarray*}
This completes the proof of Theorem \ref{prop:b}.
\end{proof}

\subsection{Proof of Theorem \ref{thm:slope-cable}} 
Theorem \ref{prop:b} implies that Conjecture \ref{conj} holds true for 2-fusion knots: The fact that $b_K(n)\leq 0$ is clear by the statement
of Theorem \ref{prop:b}.  
Moreover, $b_K(n) = 0$ if and only if $m_1 \in \{0,1\}$ and $m_2 \ge 1$, or $(m_1, m_2) = (-1,1)$.
As noted in \cite{GV}  we have
$K(1, m_2)= K^*(0, -m_2-1)$. On the other hand, by definition 
the knot $K(0, m_2)$ is a torus  knot. Finally  $K(-1, 1)$ is the torus knot $T(2,5)$.
Thus if $b_K(n)=0$, and $K=K(m_1, m_2)$, then $K$ is a torus knot. \qed

\smallskip

\begin{example} 
\label{example-pretzel}
Consider the 2-fusion knot $K(m, 1)$, also known as the $(-2,3,2m+3)$-pretzel knot. It is known that $K(m,1)$ is $B$--adequate if $m \le -2$ and is $A$--adequate if $m \ge -1$. Moreover $K(m,1)$ is a torus knot if $|m| \le 1$, and $K(-2,1)$ is the twist knot $5_2$ which is an adequate knot. Hence we consider the two cases $m \ge 2$ and $m \le -3$ only. 

Note that $K(m_1,m_2)$ is the mirror image of $K(1-m_1, -1-m_2)$. In particular, $K(m,1)$ is the mirror image of $K(1-m, -2)$.

\textit{Case 1:} $m \ge 2$. From the proof of Theorem \ref{prop:b} we have
$$
d_+[J_{K(m, 1)}(n)] = \left( \frac{5}{2} + m + \frac{1}{4m} \right) n^2 + \left( \frac{1}{2m} - \frac{1}{2} \right) n - \left( 3 + \frac{3m}{4} - \frac{1}{4m} \right) - m r^2_{n-1}
$$
where $r_n$ is a periodic sequence with $|r_n| \le 1/2$.

Hence, by Theorem \ref{thm-quasi-constant}, the $(p,q)$-cable of $K(m,1)$ satisfies Conjecture \ref{conj} and the Slope Conjecture if $$p-\left( 10 + 4m + \frac{1}{m} \right)q <0 \quad \text{or} \quad p-\left( 10+ 4m + \frac{1}{m} \right)q > \frac{m}{4} + \frac{1}{m} - 1.$$

\textit{Case 2:} $m \le -3$. From the proof of Theorem \ref{prop:b} we have
$$d_+[J_{K(1-m, -2)}(n)] = -\frac {2(m+2)^2} {2m+3} \, n^2 + b(n) n + (6m+17)/8 + (m + 3/2) r^2_{n-1}$$
where $r_n$ is a periodic sequence with $|r_n| \le 1/2$, and $b(n) = \begin{cases} -\frac{1}{2}   &\mbox{if } (2m+3) \nmid n, \medskip \\ 
-\frac {2m+1} {2( 2m+3)} & \mbox{if } (2m+3) \mid n. \end{cases}$

\smallskip

Hence, by Theorem \ref{thm-quasi-constant}, the $(p,q)$-cable of $K(1-m, -2) = (K(m,1))^*$ satisfies Conjecture \ref{conj} and the Slope Conjecture if $$p + \left( 10 + 4m + \frac{1}{2m + 3} \right)q < 0 \quad \text{or} \quad p + \left( 10+ 4m + \frac{3}{2m+3} \right)q > - \left( \frac{m}{4} + \frac{1}{2m+3}+ \frac{11}{8} \right).$$
\end{example}

\no{ Similarly, let $K$ be a knot such that for $n \gg 0$ we have
$$
d_-[J_K(n)]=a^*n^2+b^*(n)n+d^*(n)
$$
where $a^*$ is a constant, $b^*(n)$ and $d^*(n)$ are periodic functions, and $b^*(n) \ge 0$. Let
\begin{eqnarray*}
M^*_1 &=&  \max\{|b^*(i)-b^*(j)|\},\\
M^*_2 &=& \max\{-2b^*(i) + |b^*(i) - b^*(j)| +|d^*(i)-d^*(j)|\}.
\end{eqnarray*}

Suppose $p - (4a^* + M^*_1)q > 0$, or $p - (4a^* - M^*_1)q < - M^*_2$ and $p - 4a^*q <0$. Then for $n \gg 0$ we have
$$d_-[J_{\K}(n)]=A^*n^2+B^*(n)n+D^*(n)$$
where $A^*$ is a constant, and $B^*(n), D^*(n)$ are periodic functions, and $B^*(n) \ge 0$. Moreover, if $js^*_K \subset bs_K$ then $js^*_{\K} \subset bs_{\K}$.
}

\bibliographystyle{plain} \bibliography{biblio}
\end{document}